\documentclass[a4paper,11pt,reqno]{article}
\usepackage[utf8]{inputenc}
\usepackage{subcaption}
\usepackage{tabularx}
\usepackage{amsmath,amssymb,amsthm}
\usepackage[english]{babel}
\usepackage{mathtools}
\usepackage{enumerate}
\usepackage{hyperref}
\usepackage[numbers]{natbib}
\usepackage{comment}
\usepackage{graphicx}
\usepackage[dvipsnames]{xcolor}
\usepackage{color}
\usepackage{tikz}
\usetikzlibrary{calc,fit,matrix,arrows,automata,positioning}

\usepackage{hyperref}
\hypersetup{
  colorlinks   = true, 
  urlcolor     = blue, 
  linkcolor    = OliveGreen, 
  citecolor   = red 
}

\usepackage{tabularx}
\usepackage{booktabs}
\usepackage{makecell}
\usepackage{longtable}
\usepackage[ruled,vlined]{algorithm2e}
\usepackage{algorithmic}
\usepackage{listings}
\usepackage{etoolbox}
\usepackage{lipsum}
\usepackage[left =2.7cm, right = 2.7 cm]{geometry}

\usepackage{xparse}

\ExplSyntaxOn
\NewDocumentCommand{\cycle}{ O{\;} m }
 {
  (
  \alec_cycle:nn { #1 } { #2 }
  )
 }

\seq_new:N \l_alec_cycle_seq
\cs_new_protected:Npn \alec_cycle:nn #1 #2
 {
  \seq_set_split:Nnn \l_alec_cycle_seq { , } { #2 }
  \seq_use:Nn \l_alec_cycle_seq { #1 }
 }
\ExplSyntaxOff

\usepackage{url}

\usepackage{amssymb}
\usepackage{amsfonts}
\usepackage{amsmath}
\usepackage{comment}
\usepackage{tabularx}
\usepackage{booktabs}
\usepackage{makecell}
\usepackage{longtable}
\usepackage{graphicx}
\usepackage{caption}
\usepackage{lipsum}
\usepackage{xcolor}
\usepackage{parskip}
\usepackage{caption}
\usepackage{floatrow}

\usepackage{hyperref, enumitem}
\DeclareMathOperator{\Aut}{Aut}
\DeclareMathOperator{\rk}{rk}

\DeclareMathOperator{\PG}{PG}

\newcommand{\Z}{\mathbb{Z}}
\newcommand{\Zn}{\mathbb{Z}_n}
\newcommand{\F}{\mathbb{F}}

\newcommand{\mC}{\mathcal{C}}

\newcommand{\mI}{\mathcal{I}}

\newcommand{\mM}{\mathcal{M}}
\newcommand{\mB}{\mathcal{B}}

\newcommand{\Orb}{\mathrm{Orb}}
\newcommand{\Stab}{\mathrm{Stab}}
\newcommand{\sh}{\mathrm{sh}}
\DeclarePairedDelimiter\card{\lvert}{\rvert}
\newfloatcommand{capbtabbox}{table}[][\FBwidth]

\renewcommand{\a}{\alpha}

\newcommand{\fqn}{{\mathbb{F}_{q^n}}}
\newcommand{\fq}{{\mathbb{F}_q}}

\newtheorem{theorem}{Theorem}

\newtheorem{problem}{Problem}

\newtheorem{lemma}{Lemma}
\newtheorem{proposition}{Proposition}
\newtheorem{corollary}{Corollary}
\theoremstyle{definition}

\newtheorem{definition}{Definition}
\theoremstyle{remark}
\newtheorem{remark}{Remark}
\newtheorem{example}{Example}

\title{On Cyclic Matroids and their  Applications}
\usepackage{authblk}
\author[1]{Gianira N. Alfarano}
\affil[1]{Institute of Mathematics, University of Zurich, Switzerland}

\author[2]{Karan Khathuria}
\affil[2]{Institute of Computer Science, University of Tartu}

\author[1]{Simran Tinani}

\date{}

\begin{document}

\maketitle

\begin{abstract}
A matroid is a combinatorial structure that captures and generalizes the algebraic concept of linear independence under a broader and more abstract framework. Matroids are closely related  with many other topics in discrete mathematics, such as graphs, matrices, codes and projective geometries. In this work, we define cyclic matroids as matroids over a ground set of size $n$ whose automorphism group contains an $n$-cycle. We study the properties of such matroids, with special focus on the minimum size of their basis sets. For this, we broadly employ two different approaches: the multiple basis exchange property, and an orbit-stabilizer method, developed by analyzing the action of the cyclic group of order $n$ on the set of bases. We further present some applications of our theory to algebra and geometry, presenting connections to cyclic projective planes, cyclic codes and $k$-normal elements.
\end{abstract}

\section{Introduction}

Matroids are versatile combinatorial structures known to have close ties with other objects in discrete mathematics, such as graphs, matrices, codes and projective geometries.
In this paper, we introduce the family of \emph{cyclic matroids} as matroids over a ground set of size $n$ whose automorphism group contains a cyclic subgroup of order $n$. We show that these matroids are highly pertinent to the study of cyclic projective planes, cyclic codes and $k$-normal elements over finite fields. 

We describe the properties of cyclic matroids, focusing our attention on the size of their basis sets. Counting the bases of matroids is a common problem in combinatorics and in particular in optimization; see  \cite{lomeli1996randomized, gimenez2005number, snook2012counting, pendavingh2018number, guo2021approximately} to mention but a few contributions. In full generality, giving an estimate of the number of bases of a matroid is a difficult problem. Indeed, its exact computational complexity is still only partially understood. Depending on the class of matroids in question, the exact counting problem may be polynomial time, NP-complete or unsolved. For example, it is NP-complete to count the number of bases of transversal matroids and bicircular matroids \cite{colbourn1995complexity, gimenez2006complexity}.

We show that the defining feature of a cyclic matroid enforces the presence of certain types, and therefore a threshold number, of basis elements. Let $\mM$ be a cyclic matroid with ground set $\{0,1,\ldots, n-1\}$ and rank $k$. We prove, in particular, that the subset $B_0=\{0,1,\ldots,k-1\}$ is always a basis for $\mM$. Further, we provide some lower bounds on the number of bases of $\mM$. For this purpose, we employ two different approaches. On the one hand we use the \emph{basis exchange property} on the cyclic shifts of $B_0$, i.e., $B_i=\{i,i+1,\ldots,i+k-1\}$. On the other hand, we use the fact that the basis set $\mathcal{B}$ is closed under the action of the cyclic group $\Z_n$ of order $n$, and find the minimum number of orbits contained in $\mathcal{B}$. 
We also study the action of $\Z_n$ on the power set $2^{\Z_n}$ in more generality and provide a formula for the number of orbits. 
While this group action has been studied for different purposes (see for instance \cite{alon1989combinatorial, radcliffe2006reconstructing, simon2018combinatorial, mnukhin1992k}), to the best of our knowledge, neither the questions we address about it, like the exact number of orbits, nor its connection to matroids, has been investigated before.

 We finally describe the connections of cyclic matroids to well-known structures in algebra and geometry. We observe that every cyclic code of length $n$ and dimension $k$ gives rise to a representable cyclic matroid of rank $k$ and ground set $\{0,1,\dots,n-1\}$ and, more generally, the incidence matrix of every cyclic projective plane $\PG(2,q)$ can be represented as a binary cyclic matroid over a ground set of size $q^2+q+1$, and rank depending on $q$. Furthermore, we  establish and explain a new connection between $k$-normal elements of $\F_{q^n}$ and cyclic matroids of rank $(n-k)$ and ground set $\{0,1,\ldots,n-1\}$. 
However, all these connections leave open a lot of questions; in particular it is not clear yet if representable cyclic matroids are always represented by cyclic codes or $k$-normal elements. We leave those problems for further investigation. 

 The paper is structured as follows. In Section \ref{sec:background}, we give some preliminary background on matroids, finite projective geometry, linear codes, and $k$-normal elements. In Section \ref{sec:cyclic-matroids}, we introduce cyclic matroids and study the structure and properties of their basis sets. To this end, we use the multiple basis exchange property to provide different lower bounds on the number of bases of cyclic matroids. Then, we study the orbits and the stabilizers of the group action of $\Z_n$ on the set $2^{\Z_n}$, and then apply this analysis to the bases of cyclic matroids, to obtain another lower bound. Since the formulas for these bounds are obtained using different approaches, they are not directly comparable, and one may exceed the other depending on the relationship between $n$ and $k$. Therefore, we provide and compare some concrete values of all of our calculated bounds for different values of $n$ and $k$.
 In Section \ref{sec:examples}, we explain the link between cyclic projective planes, cyclic codes and $k$-normal elements and cyclic matroids. These links indicate that the class of cyclic matroids deserves to be further studied from different points of view, and may hold powerful potential to uncover results on related algebraic, combinatorial and geometric objects. In Section \ref{sec:numberoforbits}, we study the cyclic group action of $\Z_n$ on the set $2^{\Z_n}$ in generality, and characterize and exactly count the orbits. Finally, in \ref{appendix:examples} we provide some concrete, non-trivial examples of cyclic matroids found by computer search.

\subsection*{Notation}
Let $E$ be a finite set. We denote with $2^E$ the set of all subsets of $E$. The cardinality of $E$ is denoted by $|E|$. Let $A\subseteq E$ be any subset. We denote by $A^C$ the complement set of $A$ in $E$.
Given a positive integer $n$, let $\Zn$ denote the set of integers modulo $n$. By abuse of notation, we will use the same symbols for the integer $a$ and its residue class $a \mod n$, and perform arithmetic modulo $n$, unless stated otherwise. As is standard, the elements of $\Z_n$ are represented by integers $0\leq a\leq n-1$, and so $\Z_n$ inherits the ordering on $\Z$. We denote by $\mathcal{S}_n$ the symmetric group on $n$ symbols, and by the cycle $\cycle{0, 1,  \ldots, n-1}$ the permutation $ 0\mapsto 1, \; 1\mapsto 2, \; \ldots, \; (n-2)\mapsto (n-1), \; (n-1)\mapsto 0$.

\section{Background} \label{sec:background}
In this section, we provide some useful background for the rest of the paper. We first briefly recall what a matroid is; later we introduce other combinatorial and algebraic structures, such as projective planes, linear codes and $k$-normal elements, which are closely related to the main object of study of this paper, namely \emph{cyclic matroids}. These relations will be explained in Section~\ref{sec:examples}.

\subsection{Matroids}
 We first recall some basic definitions of matroid theory that we will be using throughout the paper. For a detailed treatment on matroids we refer the interested reader to \cite{oxley2006matroid}. 
 
 \begin{definition}
 A \textbf{matroid} $\mM$ is a pair $(E,\mI)$, where $E$ is a finite set and $\mI$  is a subset of $2^E$ satisfying the following properties:
\begin{itemize}
    \item[(I1)] $\emptyset \in \mI$.
    \item[(I2)] If $I \in \mI$ and $J \subseteq I$, then $J \in \mI$.
    \item[(I3)] If $I,J \in \mI$ and $|I| < |J|$, then there is an element $e \in J \setminus I$ such that $I \cup \{e\} \in \mI$.
\end{itemize}
The elements of $\mI$ are called the \textbf{independent sets} of $\mM$, and the elements outside $\mI$ are called the \textbf{dependent sets} of $\mM$. A maximal (with respect to inclusion) independent set in $\mI$ is called a \textbf{basis} of $\mM$. 
 \end{definition}
 
 Let $\mB$ is the set of bases of $\mM$, then by (I1) it follows that
\begin{itemize}
    \item[(B1)] $\mB \neq \emptyset$. 
\end{itemize} 
By (I3), it is easy to see that all the bases have the same cardinality, which is called the \textbf{rank} of $\mM$. Moreover, the set of bases $\mB$ satisfies the following property, known as the \emph{basis exchange property}:
\begin{itemize}
    \item[(B2)] If $B_1, B_2 \in \mB$ and $x \in B_1 \setminus B_2$, then there exists an element $y \in B_2 \setminus B_1$ such that $(B_1 \setminus \{x\}) \cup \{y\} \in \mB$. 
\end{itemize}
Using $(B1)$ and $(B2)$, we get an equivalent characterization of a matroid in terms of bases. 
Throughout the paper we define matroids using the bases axioms, and use the notation $\mM=(E, \mB)$ for a matroid with ground set $E$ and basis set $\mB$. 

\begin{example}[Uniform Matroid]\label{ex:uniform_matroid}
Let $E$ be a set of cardinality $n$ and for an integer $k\leq n$ denote by $\mI$ the collection of subsets of $E$ with at most $k$ elements and by $\mB$ the collection of subsets of $E$ with exactly $k$ elements. It is not difficult to verify that $\mI$ satisfies the properties (I1)--(I3) and $\mB$ satisfies the properties (B1)--(B2), hence then the pair $(E,\mB)$ defines a matroid of rank $k$, denoted by $U_{n,k}$ and called \textbf{uniform matroid}.
\end{example}
Below, we state an equivalent formulation of the basis exchange property given in \cite{greene1973multiple}, which we will use many times in the paper.

\begin{lemma}[Multiple Exchange Property] \label{lem:basis_exch} Let $\mM= (E,\mB)$ be a matroid on a ground set $E$ and $\mB$ be its collection of bases. Further, let $B_1, B_2 \in \mathcal{B}$, and $Q \subset B_1 \setminus B_2$. Then there exists a subset $P \subset B_2 \setminus B_1$ such that $(B_1 \setminus Q) \cup P \in \mathcal{B}$.
\end{lemma}

\begin{definition}
Let $\mM = (E,\mB)$ be a matroid on the ground set $E$ with bases set $\mB$. An \textbf{automorphism} $\tau$ of $\mM$ is a permutation of $E$ such that $B \in \mB$ if and only if $\tau(B) \in \mB$, where $\tau(B):=\{\tau(b)\mid b\in B\}$. The \textbf{automorphism group} $\Aut(\mM)$ is the group of automorphisms of $\mM$ under composition.
\end{definition}

More generally, given a matroid $\mM = (E,\mB)$, another matroid $\mM' = (E',\mB')$ is said to be \textbf{isomorphic} to $\mM$ if there exists a bijection $\tau : E\rightarrow E'$ such that $\tau(B) \in \mB'$ if and only if $B \in \mB$.

Finally, we recall the notion of dual matroid.

\begin{definition}
Let $\mM=(E,\mB)$ be a matroid with ground set $E$ and collection of bases $\mB$. Let $\mB^\ast=\{B^C \mid B\in\mB\}$. Then $\mB^\ast$ satisfies the axioms (B1) and (B2), hence it is the collection of bases of a matroid $\mM^\ast=(E,\mB^\ast$), called the \textbf{dual matroid} of $\mM$.
\end{definition}

\begin{example}[Representable Matroid]\label{rep-mat}
Let $\F$ be a field, and $A$ be an $m \times n$ matrix over $\F$. We define $E$ to be the index set of the columns of $A$, and $\mI$ to be the collection of subsets of $E$ that correspond to linearly independent sets of columns of $A$. Then, $\mM(A) = (E,\mI)$ is a matroid of rank equal to the rank of $A$, and it is called \textbf{representable matroid}.  A proof can be found in \cite[Theorem 1.1.1]{oxley2006matroid}.
\end{example}

\subsection{(Cyclic) Projective Planes}\label{subsec:projplanes}
In this short subsection, for convenience of the reader, we recall the definition of a (cyclic) projective plane, incidence matrix and collineation group. 

\begin{definition}
A  (point-line) \textbf{incidence structure} is a triple $(\mathcal{P}, \mathcal{L}, \mathcal{I}$) of sets, with $ \mathcal{P}$ called \textbf{points}, $\mathcal{L}$ called \textbf{lines}, and $\mathcal{I}\subseteq \mathcal{P} \times \mathcal{L}$ called the \textbf{incidence relations}. We say that a point $p$ and a line $\ell$ are \textbf{incident} with each other if $(p, \ell)\in \mathcal{I}$, and in this case we write $p\in \ell$. A subset of points is called \textbf{collinear} if they are all incident with the same line.
\end{definition}
\begin{definition}
Let $n$ be an integer.  A  \textbf{projective plane} $\PG(2,n)$ is a (point-line) incidence structure with $n^2+n+1$ points and $n^2+n+1$ lines which satisfy the following axioms:
\begin{enumerate}
    \item Every two points are incident with exactly one line.
    \item Every two lines are incident with exactly one point. 
    \item  There are four points such that no three of them are collinear.
\end{enumerate}
When $n=q$, where $q$ is a prime power, then the points and lines of $\PG(2,q)$ are the one- and two-dimensional subspaces of a vector space of dimension $3$ over the finite field of order $q$, $\F_q$. In this case, the projective plane is called \textbf{Desarguesian}.
\end{definition}

We are only interested in the Desarguesian plane $\PG(2,q)$. It is immediate to see that every line in $\PG(2, q)$ is incident with exactly $q + 1$ points and dually, every point is incident with exactly $q + 1$ lines. The incidence relation of $\PG(2,q)$ can be represented via an \textbf{incidence matrix $A$}, whose rows and columns are indexed by points and lines respectively such that 
$$A_{i,j} = \begin{cases}
1 & \textnormal{ if } P_i\in \ell_j \\
0 & \textnormal{ otherwise },
\end{cases}$$
where for $i,j \in \{1,\dots, q^2+q+1\}$, the $P_i$'s are the points and $\ell_j$'s are the lines of the projective plane. 

\begin{definition}
A \textbf{collineation} of $\PG(2,q)$ is a permutation of the points of $\PG(2,q)$ which preserves their collinearity, i.e. lines are mapped onto lines. The set of collineation forms a group, called \textbf{collineation group}. 
\end{definition}

\begin{definition}
A projective plane $\PG(2,q)$ is called \textbf{cyclic} if its collineation group is transitive on the points of $\PG(2,q)$ and there exists a collineation that generates a cyclic subgroup of order $q^2+q+1$.
\end{definition}

\subsection{Linear Codes}\label{subsec:codes}
This subsection introduces linear codes, with a particular focus on cyclic codes. For a more detailed treatment of the topic we refer the interested reader to \cite{van2012introduction}. Let $n$ be a positive integer, $q$ be a prime power and $\F_q$ be the finite field of order $q$. Let $\F_q^n$ be the $n$-dimensional vector space over $\F_q$.

\begin{definition}
An $[n,k]_q$ \textbf{(linear) code} $\mC$ is a $k$-dimensional $\F_q$-subspace of $\F_q^n$. The vectors in $\mC$ are called \textbf{codewords}. A matrix $G\in\F_q^{k\times n}$ whose rows form a basis for $\mC$ is called a \textbf{generator matrix} for $\mC$. 
\end{definition}

In 1976, Greene \cite{greene} explored several connections between matroids and linear codes. Ever since, many authors have exploited this link and used matroid theory to prove coding theoretic results. It is straightforward to obtain a representable matroid from a linear code $\mC$: given a generator matrix $G$ of $\mC$, we obtain a representable matroid $\mM(G)$ (see Example~\ref{rep-mat}). Note that $\mM(G)$ does not depend on the choice of the generator matrix $G$.

In this work, we are interested in a special class of linear codes, called \emph{cyclic codes}, which are one of the most studied families of codes due to their polynomial representation as ideals of $\F[x]/\langle x^n-1\rangle$. More precisely, they are defined as follows.

\begin{definition}
A code $\mC\subseteq \F_q^n$ is said to be \textbf{cyclic} if for every codeword $c=(c_0,c_1,\dots,c_{n-1})$, also the cyclic shift of $c$, namely $\mathrm{sh}(c)=(c_{n-1},c_0,\dots, c_{n-2})$, is a codeword.
\end{definition}

Consider the following map:
\begin{align*}
    \phi : \F_q^n &\rightarrow \F_q[x]/\langle x^n-1 \rangle \\
    (c_0,\dots, c_{n-1}) &\mapsto c_0+c_1x+\cdots+c_{n-1}x^{n-1}.
\end{align*}
It is easy to see that $\phi$ is an isomorphism of vector spaces and it turns out that $\mC\subseteq\F_q^n$ is a cyclic code if and only if $\phi(\mC)$ is an ideal of $\F_q[x]/\langle x^n-1 \rangle$ (which derives from the fact that $\phi(\sh(c))=x\phi(c)\mod (x^n-1)$). With abuse of notation, we then identify $\mC$ with $\phi(\mC)$ and we say that a cyclic code is an ideal of
$\F_q[x]/\langle x^n-1 \rangle$.

Since $\F_q[x]/\langle x^n-1 \rangle$ is a principal ideal ring, every cyclic code consists
of the multiples of a polynomial $g(x)$ which is the monic polynomial of lowest degree in the ideal. Such a polynomial $g(x)$ is called \textbf{generator polynomial}, it divides $x^n-1$ and if $g(x)$ has degree $n-k$, then the cyclic code that it generates has dimension $k$.

\subsection{$k$-Normal Elements}\label{subsec:knorm}

In this last introductory subsection, we introduce $k$-normal elements. 

Let $q$ denote a prime power, and $\fq$ denote the finite field of order $q$. We are interested in studying elements in a finite extension $\fqn$ of degree $n$ over $\fq$. An element $\a \in \fqn $ is called a \textbf{normal element} over $\fq$ if all its Galois conjugates, i.e. the $n$ elements $\{\alpha, \alpha^q, \ldots, \alpha^{q^{n-1}}\} $, form a basis of $\fqn$ as a vector space over $\fq$. A basis of this form is called a \textbf{normal basis}.

As a generalization of normal elements, in \cite{huczynska2013existence} \emph{$k$-normal elements} were defined.

\begin{definition} An element $\a \in \fqn$ is called \textbf{$k$-normal} if \[\dim_\fq \left(\mathrm{span}_\fq \left\{\alpha, \alpha^q, \ldots, \alpha^{q^{n-1}} \right\} \right) = n-k.\] 
\end{definition}

Questions on the existence of $k$-normal elements have been investigated in \cite{reis2019existence} and in \cite{tinani_rosenthal_2021}. In this last work, a general lower bound for the number of $k$-normal elements has been also provided.

\section{Cyclic Matroids}\label{sec:cyclic-matroids}
In this section we introduce cyclic matroids and study the structure of their basis sets.

\begin{definition}\label{def:cyclic_mat}
Let $n$ be a positive integer. A matroid $\mM = (E,\mB)$ on the ground set $E$ with $|E|=n$ is called a \textbf{cyclic $k$-matroid} if it has rank $k$ and satisfies one of the following equivalent conditions \begin{enumerate}
\item Its automorphism group $\Aut(\mM)$ contains an isomorphic copy of the cyclic group $\Z_n$.
    \item There exists a cycle $\sigma$ of length $n$ acting on $E$ such that $\sigma(B) \in \mB$ for each $B \in \mB$.
    \item $\mM$ is isomorphic to some matroid $\mM_0 $ with ground set $\{0,1,\ldots, n-1\}$ satisfying \\ $\cycle{0, 1, \ldots, n-1} \in \Aut(\mM_0)$.
\end{enumerate}
When the rank is clear or not necessary, we simply say that $\mM$ is a \textbf{cyclic matroid}.
\end{definition}

For simplicity, we fix the ground set $E= \{0,1,\ldots,n-1\}$, and use $\Z_n$ interchangeably with $E$. Without loss of generality, we can assume that the automorphism group of a cyclic matroid contains the $n$-cycle $\cycle{0,1,\ldots,n-1}$. We define \textit{cyclic shifts} \textbf{on the subsets} of $E$ as follows: let $s \in \Zn$ and $A \subseteq E$, then the shifted subset $s + A$ is defined as $\sigma^s(A)$, where $\sigma$ is the permutation $\cycle{0,1,\ldots,n-1}$. Note that if $A = \{g_0,g_1\ldots,g_{k-1}\}$, then $s+A = \{s + g_0 \bmod n, s+g_1 \bmod n, \ldots, s+g_{k-1} \bmod n\}$.

Notice that  in \cite[p. 330]{welsh2010matroid}, Welsh defines a matroid $\mM$ to be cyclic if $\Aut(\mM) = \Z_n$. We redefine cyclic matroids because, in theory, cyclic objects are defined to be the ones which are closed under cyclic shifts. Definition \ref{def:cyclic_mat} introduces a more general class of matroids which strictly include Welsh's cyclic matroids. For example, the uniform matroid $U_{k,n}$ defined in Example~\ref{ex:uniform_matroid} has automorphism group equal to the symmetric group $\mathcal{S}_n$, hence it is not cyclic according to Welsh's definition. Another example is the well-known Fano matroid, which will be described in  Example~\ref{fano}.

\begin{remark}
For fixed values of $n$ and $k$, cyclic matroids are not uniquely determined. Indeed, for example, in the case of $n=4$ and $k=2$, we have two distinct cyclic $k$-matroids with bases $\mB=\{\{0,1\}, \{1,2\}, \{2,3\}, \{0,3\}\}$ and $\mB^\prime=\{\{0,1\}, \{1,2\}, \{2,3\}, \{3,0\}, \{0,2\}, \{1,3\}\}$. Note that the matroid $(\{0,1,2,3\}, \mB^\prime)$ is the uniform matroid $U_{2,4}$.

\end{remark}

\begin{remark}
It is easy to see that the dual matroid $\mM^\ast$ of a cyclic $k$-matroid with ground set of size $n$, is a cyclic $(n-k)$-matroid. Moreover, for $k \geq 1$, every singleton in a cyclic matroid clearly has rank $1$. So a non-trivial cyclic matroid does not have loops (i.e, elements that do not belong to any basis), and a proper cyclic matroid does not have coloops (i.e., elements that belong to every basis). Another remarkable property of cyclic $k$-matroids is that they are connected (i.e., each pair of elements in $E$ is contained in a minimal dependent set) whenever $k\leq n/2$.
\end{remark}

 The main problem we address in this paper is counting the minimum number of basis elements, i.e. finding the minimum cardinality $\card{\mB}$ in a cyclic $k$-matroid.
Throughout, we let $B_0 := \{0,1,\ldots,k-1\}$. Given a cyclic $k$-matroid $\mM$, we show in Proposition \ref{systematic} that $B_0$ is always a basis of $\mM$. As a result, each of its shifts $B_i = i+ B_0 $ for $1 \leq i \leq n-1$, is also a basis. We will refer to these bases as \textbf{cyclic bases} for the matroid $\mM$. To prove this result, we will associate to each subset of $\Z_n$ a set partition in the following way. 

\begin{definition}\label{cons-block}
Let $A \subseteq \Z_n$ be any set. The \textbf{consecutive block structure} of $A$ is the (ordered) set partition of $A$ given by $\pi(A) = (D_1, D_2, \ldots, D_\ell)$, where each $D_i = \{d_i,d_i+1,\ldots,d_i+|D_i|-1\}$ is a maximal subset of $A$ containing consecutive elements modulo $n$, ordered according to $d_1 < d_2 < \ldots < d_\ell$. \end{definition}

It will be useful also to associate to a $k$-subset of $\Zn$ the following tuple.

\begin{definition}
If $|A| = k$, given the consecutive block structure of $A$, $\pi(A) = (D_1, D_2, \ldots, D_\ell)$, the composition (ordered integer partition) $|D_1| + |D_2| + \cdots + |D_\ell|$ of $k$ is called the \textbf{block composition} of $A$, denoted by $c(A) = (|D_1|,|D_2|,\ldots,|D_\ell|)$.
\end{definition}

\begin{example}
Let $A = \{0,2,3,4,6,7,9\} \subseteq \Z_{10}$, then the consecutive block structure of $A$ is $\pi(A) = (\{2,3,4\}, \{6,7\}, \{9,0\})$, and the block composition of $A$ is $c(A) = (3,2,2)$.
\end{example}

\begin{proposition}\label{systematic}
Let $\mM = (E,\mB)$ be a cyclic $k$-matroid. Then $B_0=\{0, 1, \ldots k-1\}$ is a basis for $\mM$. 
\end{proposition}
\begin{proof}
Let $B \in \mB$ be a basis of $\mM$, and let $\pi(B) =(D_1, D_2, \ldots, D_\ell)$ be its consecutive block structure, i.e. $B=D_1 \cup D_2 \cup \ldots \cup D_\ell$.
 If $\ell = 1$, then we are done as one of the shifts of $B$ will be equal to  $B_0$.

Now assume that $\ell > 1$. Using the basis exchange property we will construct a new basis that has $\ell-1$ number of blocks in its consecutive block structure. Note that this is enough to prove the result, as we can apply this argument repetitively until we obtain $\ell=1$. 

Let $D_i = \{d_i,d_i+1,\ldots,d_i + \card{D_i} -1 \}$ for each $i \in \{1,2,\ldots,\ell\}$. We apply the basis exchange property, Lemma \ref{lem:basis_exch}, with respect to bases $B$ and $B+1$, to obtain a new basis element $B^\prime= (B \setminus \{d_\ell\}) \cup \{p\}$, for some $p \in (B+1)\setminus B = \{d_1+\card{D_1},\ldots,d_\ell+\card{D_\ell}\}$. 

Let $B^\prime = D_1^\prime \cup \ldots \cup D^\prime_{\ell^\prime}$ be the consecutive block structure of $B^\prime$. Then it is easy to check that $\ell^\prime = \ell$ or $\ell^\prime = \ell-1$. If $\ell^\prime = \ell-1$, then we are done. So assume that $\ell^\prime = \ell$. This implies that $D^\prime_{\ell} = \{d_\ell + 1,\ldots,d_\ell+\card{D_\ell}-1\}$ or $D^\prime_\ell = \{d_\ell+1,\ldots,d_\ell+\card{D_\ell}\}$. In each case, the smallest element in $D^\prime_\ell$ increases by 1. Hence, by applying this basis exchange process repeatedly the block $D^\prime_\ell$ either vanishes or merges with the next block $D^\prime_1$, and results in a new basis with $\ell-1$ consecutive blocks. 
\end{proof}

\subsection{Basis Exchange Approach for the Number of Bases}\label{number_bases}

 In this subsection, we present some lower bounds on the size of the collection of bases $\mathcal{B}$ of an arbitrary cyclic matroid $\mM=(E,\mB)$. In particular, we use the basis exchange property on the cyclic bases $B_0, B_1, \ldots, B_{n-1}$ to construct other basis elements. 

We first prove some properties of the cyclic bases that we will use to count the number of bases in a cyclic matroid. 
In the following result we compute the size of the difference between the intersection of two cyclic bases and the basis set $B_0$.   

\begin{lemma}\label{intersectioncontained} Let $1 \leq j \leq i \leq n-1$. Then, \[\left| (B_i \cap B_j) \setminus B_0   \right| = \begin{cases}
0 & \mbox{if } i-j\geq k \\
j + k- \max(k,i)& \textnormal{otherwise.}
\end{cases}\]

\end{lemma}

\begin{proof}

Write $\ell = i-j$, then it is easy to check that 
   \begin{equation}
      B_i \cap B_j = \begin{cases} \{i, i+1, \ldots, j+k-1\} & \ \text{if} \ \ell<k, \  \ell \leq n-k  \\
     \{j, j+1, \ldots, i-n+k-1\} & \ \text{if} \ \ell \geq k, \ \ell >n-k \\
     \{j, j+1, \ldots, i-n+k-1\}\cup\{i,i+1, \ldots, j+k-1\} & \ \text{if} \ \ell<k, \ \ell > n-k \\
     \emptyset, & \ \text{otherwise}.
     \end{cases}
 \end{equation}

Therefore, if $\ell \geq k$, then $B_i \cap B_j = \emptyset$ or $B_i \cap B_j \subseteq B_0$. Whereas, if $\ell < k$ then
\[(B_i \cap B_j) \setminus B_0 = \begin{cases} \{k, k+1, \ldots, j+k-1\} & \ \text{if}  \ i \leq k  \\
   \{i, i+1, \ldots, j+k-1\} & \ \text{if} \ i > k.
    \end{cases}
    \]
\end{proof}
In order to apply the basis exchange property on $B_0$, we calculate the collection of all cyclic bases that intersect trivially with a subset of $B_0$.

\begin{lemma}
Let $Q \subseteq B_0$. Let $q_1$ and $q_2$ denote, respectively, the smallest and largest elements of $Q$. For $i \in \{ q_2-q_1 +1, \ldots, n-k\}$, the $n-k-q_2+q_1$ bases $B_0+q_1+i$ intersect trivially with $Q$, and any basis satisfying this property must lie among these.
\label{lem:basis_trivial_intersect}
\end{lemma}
\begin{proof} We may rule out cyclic bases of the form $B_0 + j$ with $j<q_2$, since these would always contain $q_2$. So, we are looking for cyclic bases of the form $B_0+q_1+i$ with $i>q_2-q_1$. For $q_1$ to lie outside these bases, we would additionally need $k-1 + q_1+i < n+q_1$, so $i\leq n-k$. Now let $q \in Q$, and suppose that $q \in B_0 + q_1+i$ for some $q_2-q_1<i\leq n-k$. Thus, $q = b_0 + q_1 + i$ or $q+n = b_0 +q_1 + i$ for some $q_2-q_1<i\leq n-k$ and $b_0 \in B_0$. Since $q < q_2$, the first case is impossible. Similarly, $b_0 +q_1+i < k-1+q_1+i <n+q_1 <n+q$, so the second case is also impossible. Thus, $Q \subseteq B_0\setminus B_{q_1+i} $ for all $q_2-q_1<i\leq n-k$. The number of such bases is clearly given by $n-k-q_2+q_1$, which is the number of valid indices $i$.
\end{proof}

\begin{lemma} Let $Q \subseteq B_0$ be fixed with $\left|Q \right| = r$. Let $q_1$ and $q_2$ be the smallest and largest elements of $Q$, respectively, and assume that $q_2-q_1<n-k$.  For any $q \in B_0$ and $q_2-q_1+1 \leq j < i \leq n-k$, we have that $\card{(B_{q+i} \setminus B_0 )\cap (B_{q+j} \setminus B_0 )} <r $ if and only if $ \ i-j \geq k-r+1$ .  \label{lem:intersection_size}
\end{lemma}
\begin{proof} Let $q_2-q_1+1 \leq j < i \leq n-k$ and $i-j \geq k-r+1$. Note first that $(B_{q+i} \setminus B_0 )\cap (B_{q+j} \setminus B_0 ) = (B_{q+i}\cap B_{q+j}) \setminus B_0 $. Thus, by Lemma \ref{intersectioncontained}, we have
\[\card{(B_{q_1+i} \setminus B_0) \cap (B_{q_1+j} \setminus B_0) } = \begin{cases} 
0 & \mbox{if } i-j \geq k \\
q_1+j + k- \max(k,q_1+i)& \mbox{otherwise}.
\end{cases}\]
Now, since $r \leq q_2-q_1+1 \leq j <i$, the smallest value of $j$ and $i$, respectively, in the above expressions are $q_2-q_1+1$ and $q_2-q_1+k-r+2$, respectively. Thus, $q_1+i\geq (q_2-r)+k+2\geq q_1+1 + k >k$, and so for any $q\in B_0$ we have
 $\max\{k,q+i\}= q_1+i$ for the relevant values of $i$. Now, if $i-j \geq k$, we are done. If $i-j < k$, then, $\card{(B_{q+i} \setminus B_0) \cap (B_{q+j} \setminus B_0) } < r,$
 which can happen if and only if
 $q+j+k-q-i <r$ and  equivalently when $i-j > k-r $.  This completes the proof.
\end{proof}

For an arbitrary cyclic $k$-matroid $\mM=(\Z_n,\mB)$, we know from Proposition~\ref{systematic} that all the $n$ shifts of $B_0$ are bases of $\mM$. In the rest of the section, we use Lemma \ref{lem:basis_exch} to show the existence of more bases in~$\mB$.  

\begin{proposition}
Let $\mM=(\Z_n,\mB$) be a cyclic $k$-matroid and let $Q \subseteq B_0 = \{0,1,\ldots,k-1\}$ with $\left|Q \right| = r$. Let $q_1$ and $q_2$ denote, respectively, the smallest and largest elements of $Q$ and assume that $q_2-q_1<n-k$. Define $m=  \left\lfloor \frac{n-k-q_2+q_1-1}{k-r+1} \right\rfloor+1$. Then, there exist $m$ distinct bases in $\mB$ of the form $(B_0 \setminus Q) \cup P_{i}$, where $P_{i} \subseteq B_{q_1+i}\setminus B_0$ and $q_2-q_1 < i \leq n-k$. \label{prop:basis_exch_bound_simple}
\end{proposition}
\begin{proof} 
By Lemma \ref{lem:basis_trivial_intersect}, we have $Q \cap B_{q_1+i} = \emptyset$ precisely for the $n-k-q_2+q_1$ values of $i \in I:=\{q_2-q_1+1, \ldots, n-k\}$. Applying Lemma \ref{lem:basis_exch} to $B_0$ and $B_{q_1+i}$ for these values of $i$, we get subsets $P_i \subseteq B_{q_1+i}\setminus B_0$ such that $(B_0\setminus Q)\cup P_i$ is a basis.

Using the result of Lemma \ref{lem:intersection_size} for $q=q_1$, we have $\card{(B_{q_1+i} \setminus B_0 )\cap (B_{q_1+j} \setminus B_0 )} <r $ for each pair $i,j \in I$ that satisfies $i-j \geq k-r+1$. Thus, for each 
$$ i \in \{{q_2-q_1+1}, {q_2-q_1+1+(k-r+1)}, \ldots, {q_2-q_1+1+ \tilde{m}  (k-r+1)} \},$$
where $\tilde{m} = \left\lfloor\frac{n-k -q_2+q_1-1}{k-r+1}\right\rfloor$ is the largest integer such that $q_2-q_1+1 + \Tilde{m} (k-r+1) \leq n-k$, we get distinct bases $(B_0\setminus Q)\cup P_i$. Thus, there are $m = \tilde{m}+1$ bases in  $\mB$ of this form. 
\end{proof}

We now apply Proposition \ref{prop:basis_exch_bound_simple} on each subset $Q$ of $B_0$ to obtain a lower bound on the size of $\mB$. For the next two theorems, we use the following convention for binomial coefficients: \begin{equation}\label{binom_convention}
    \binom{a}{-1} := \begin{cases} 0 & \mbox{if}\ a \geq 0 \\ 1 & \mbox{if}\ a = -1. \end{cases}
\end{equation}

\begin{theorem}\label{thm:m1bound} Let $\mM = (\Z_n,\mB)$ be a cyclic $k$-matroid. Then, there are 
at least $m_1(n,k)$ distinct bases in $\mathcal{B}$ of the form $(B_0\setminus Q) \cup P$, where $Q\subseteq B_0$, $P \subseteq B_i$ for some $1 \leq i \leq n-1$, and
\begin{align*}    m_1(n,k)= 1+ \sum\limits_{\Delta=0}^{\min(k-1,n-k-1)} \sum\limits_{r=1}^{\Delta+1}  (k-\Delta) \binom{\Delta-1}{r-2} \left( \left \lfloor \frac{n-k-\Delta-1}{k-r+1} \right \rfloor +1\right).  
\end{align*} \label{bound:using_differences} 
\end{theorem}
\begin{proof} 
First note that for distinct subsets $Q,{Q^\prime}$ of $B_0$, bases of the form $(B_0\setminus Q) \cup P$ and $(B_0 \setminus Q^\prime)\cup P^\prime$ are distinct regardless of $P$ and $P^\prime$. So, we may simply add up the number of bases resulting from the individual subsets $Q$.

Now for a subset $Q$ of $B_0$ with smallest and largest terms $q_1$ and $q_2$, respectively, and size $r\geq 1$, write $\Delta=q_2-q_1$ as the value corresponding to $Q$. 

In the case of $\Delta = 0$, we get $r = 1$ and the number of bases of the form $(B_0\setminus \{q\})\cup \{p\}$ is given by $k \left( \left\lfloor \frac{n-k-1}{k} \right\rfloor + 1 \right)$. This directly follows from Proposition \ref{prop:basis_exch_bound_simple} by taking $q_1=q_2=q$ and $r=1$. 

 For a fixed value of $\Delta \in \{1, 2, \ldots, \min\{k-1, n-k-1\}\}$, we may calculate the number of subsets $Q$ corresponding to $\Delta$ and with a fixed size $r$ as follows. There are $(k-\Delta)$ subsets of $B_0$ of the form $\{q_1, q_1+1, \ldots, q_1+\Delta\}$, each with size $\Delta+1$. Any $Q$ must contain $q_1$ and $q_1+\Delta$, and may then contain any $(r-2)$-subset of the remaining $\Delta+1-2 =\Delta-1$ elements of this set. Thus, this gives us a total of $(k-\Delta) \binom{\Delta-1}{r-2}$ options for $Q$ corresponding to $r$.

We may further sum over the relevant values of $r$ for a given value of $\Delta$, i.e. from 2 to $\Delta+1$. For each of these subsets $Q$, there are at least $\left\lfloor \frac{n-k-\Delta-1}{k-r+1} \right\rfloor +1 $ distinct bases in $\mB$, by Proposition \ref{prop:basis_exch_bound_simple}. Finally, we add $1$ to include the case $Q=\emptyset$. This completes the proof. 
\end{proof}

The next result provides a different bound on the number of bases of a cyclic matroid.

\begin{theorem}\label{thm:m2bound} Let $\mM = (\Z_n,\mB)$ be a cyclic $k$-matroid. Then, there are 
at least $m_2(n,k)$ distinct bases in $\mathcal{B}$ of the form $(B_0\setminus Q) \cup P_{i_1 k}\cup P_{i_2 k} \cup \ldots \cup P_{i_w k} \cup P_{(\ell+1)k}$, where $P_{i_tk} \subseteq B_{i_tk} $ for $1\leq i_t \leq \ell$, $P_{(\ell+1)k} \subseteq B_{(\ell+1)k}$,  $0 \leq w \leq \ell$, $\ell = \left\lfloor\frac{n}{k} \right\rfloor-1 $, and
\begin{align*}    m_2(n,k)= \sum\limits_{\card{Q}=0}^{k}  \sum\limits_{j=0}^{\min(n-(\ell+1)k, \card{Q})}  \binom{(\ell+2)k-n}{j}\binom{k-j}{\card{Q}-j}\sum\limits_{w=0}^{\ell} \binom{\ell}{w}\binom{\card{Q}-j-1}{w-1}.
\end{align*}

\end{theorem}
\begin{proof}
Note that $\ell = \left\lfloor\frac{n}{k}\right\rfloor-1 $ gives the number of cyclic bases of the form $B_{ik}$, $1 \leq i \leq \ell$, which are disjoint pairwise  as well from $B_0$, and for which basis exchange is possible for any subset $Q \subseteq B_0$. If $k \nmid n$ then we have an additional basis $B_{(\ell+1)k}$ with $\card{B_{(\ell+1)k}\setminus B_0}=n-(\ell+1)k$. Clearly, on performing multiple basis exchanges, for every $Q \subseteq B_0$ we obtain that there must be at least one basis element of the form \begin{equation}\label{exchanged_bases}
    (B_0\setminus Q) \cup P_{i_1 k}\cup P_{i_2 k} \cup \ldots \cup P_{i_w k} \cup P_{(\ell+1)k},
\end{equation} where $P_{i_tk} \subseteq B_{i_tk} $ for $1\leq i_t \leq \ell$, $P_{(\ell+1)k} \subseteq B_{(\ell+1)k}$,  $0 \leq w \leq \ell$.
Note that any basis element can be uniquely written as in \eqref{exchanged_bases}, hence it is counted at most once.

 We count all sets of this form as follows. For a given subset $Q$, let $j$ denote the number of elements $P_{(\ell+1)k}$ picked from the last basis element $B_{(\ell+1)k}$. This leaves $\card{Q} - j$ elements to be picked from the remaining bases, which is possible in any way since they are all disjoint from $B_0$. 
 
 Let $w$ be the number of bases $B_{i_tk}$, $1 \leq t \leq w$, $1 \leq i_t \leq \ell$ chosen for basis exchange with $B_0$ after $B_{(\ell+1)k}$. For $w>0$, the number of possibilities equals to the number of compositions of $\card{Q}-j$ of length $w$, i.e. to $\binom{\card{Q}-j-1}{w-1}$. The option $w=0$ may be included in this expression as well, using our convention \eqref{binom_convention}, since it must be counted only if $\card{Q} =j$. Now since $\binom{\ell}{w}$ gives the number of choices of the bases $B_{i_jk}$ used for exchange, the total number of bases obtained by this process for a fixed subset $Q$ and fixed $j>0$ is $\sum\limits_{w=0}^{\ell} \binom{\ell}{w}\binom{\card{Q}-j-1}{w-1}$. 

Now, for fixed values of $\card{Q}$ and $j$, $Q$ can be chosen in the following way: first pick $j$ elements from $B_0 \setminus B_{(\ell+1)k} $ and then  $\card{Q}-j$ elements from the remaining $k-j$ elements in $B_0$. Thus, the total number of bases obtained by this process for a fixed cardinality $\card{Q}$ and fixed $j>0$ is $ \binom{(\ell+2)k-n}{j}\binom{k-j}{\card{Q}-j}\sum\limits_{w=0}^{\ell} \binom{\ell}{w}\binom{\card{Q}-j-1}{w-1}$. 

Finally, note that $\card{B_{(\ell+1)k}\setminus B_0} = n-(\ell+1)k $ and $\card{B_0 \setminus B_{(\ell+1)k}} = (\ell+2)k-n$. So, we must have $j \leq \min(\card{Q},  n-(\ell+1)k,  (\ell+2)k-n)$. Due to the presence of the term $\binom{(\ell+2)k-n}{j}\binom{k-j}{\card{Q}-j}$, we may do away with the last term in the minimum. $Q $ is allowed to vary across all subsets of $B_0$, so we take a sum over $0 \leq \card{Q} \leq k$. This completes the proof.
\end{proof}

\subsection{Group Action Approach for the Number of Bases}\label{gpaxnsection}

In order to further investigate cyclic matroids, we define the following group action of $\Z_n$ on $2^\Z_n$. 
\begin{align}\label{gpaxnall}
\varphi : \Z_n \times 2^{\Z_n} & \rightarrow 2^{\Z_n} \\
    (s, A) & \mapsto A+s \end{align}
 It follows from the definition that the basis set $\mB$ of a cyclic matroid is closed under the action $\varphi$. In other words, $\mB$ is a union of orbits of $2^{\Z_n}$ under $\varphi$. Therefore, in order to study some properties of a cyclic matroid, we analyse here the orbits and stabilizers of $\varphi$.

For any $A \subseteq \Z_n$, the orbit of $A$ is denoted by $\Orb(A) = \{ A + s  : s \in \Z_n \}$ and the stabilizer of $A$ is denoted by $\Stab(A) = \{s \in \Z_n : A+s = A\}$.

\begin{remark}
Let $A \subseteq \Z_n$. Then $\Stab(A) = \Stab(A^C)$, and hence $|\Orb(A)| = |\Orb(A^C)|$, where $A^C$ denotes the complement of $A$ in $\Z_n$. 
\end{remark}

\subsubsection{Size of a Stabilizer}

We know that for any $A \subseteq \Z_n$, $\Stab(A)$ is a subgroup of $\Z_n$, and so $|\Stab(A)|$ divides $n$. Moreover, $\Stab(A) = \langle s_0 \rangle$ for some $s_0$ that divides $n$.

\begin{proposition}
Let $A \subseteq \Z_n$ and $s \in \{1, \ldots, n-1\}$. Then $s \in \Stab(A)$ if and only if $A$ is a union of arithmetic progressions with common difference $s$, each with length $\frac{n}{\gcd(n,s)}$. \label{prop:AP}
\end{proposition}
\begin{proof}
First assume that $A = A_1 \cup A_2 \cup \ldots \cup A_r$, where each $A_i$ is an arithmetic progression with common difference $s$ and having length $\frac{n}{\gcd(n,s)}$. Pick a ``first term" in each $A_i$ and denote it by $a_i$. (This choice is arbitrary since we are working modulo $n$.) Note that the additive order of $s$ mod $n$ is equal to the cardinality of $A_i$, i.e. $\frac{n}{\gcd(n,s)}$, and so $a_i + j\cdot s \mod n \in A_i$ for all $j \geq 0$. In other words, we must have $A_i+s = A_i$, for every index $i\in \{1, \ldots, r\}$. Therefore, $A+s = A$.

Conversely, assume that $A+s = A$, and pick $a_1 \in A$. Again, since the additive order of $s$ mod $n$ is equal to $\frac{n}{\gcd(n,s)}$, we must have $a_1 + \frac{n}{\gcd(n,s)} \cdot s = a_1$, and $A_1:= \{a_1, a_1+s, \ldots, a_1 + \left(\frac{n}{\gcd(n,s)}-1 \right)\cdot s\} \subseteq A$. If $A=A_1$, then the proof is complete. If not, we have some $a_2 \in A\setminus A_1$, so $A_2:= \{a_2, a_2+s, \ldots, a_2 + (\frac{n}{\gcd(n,s)}-1)\cdot s\} \subseteq A$, and $A_2 \cap A_1 = \emptyset$. Continuing in this manner, we obtain, in a finite number of steps,  $A = A_1 \cup A_2 \ldots \cup A_r$, where each $A_i$ is an arithmetic progression with common difference $s$ and having length $\frac{n}{\gcd(n,s)}$. This completes the proof.
\end{proof}

In particular, the result also holds for a generator of the stabilizer group.

\begin{corollary}
Let $A \subseteq \Z_n$, $\card{A}=k$, $Stab(A)=\langle s_0\rangle$. Then for  $r=\frac{k \gcd(n,s_0)}{n} \in \Z$ we have  \begin{equation}
    A = \bigcup_{i=1}^{r} \left( a_i + \Stab(A) \right), \label{eq:AP-union}
\end{equation} 
for some $a_1,\ldots,a_r \in \Z_n$. Consequently, $\card{Stab(A)}$ divides $k$, and thus also $\gcd(k,n)$. \label{cor:AP-union}
\end{corollary}
\begin{proof}
Proposition \ref{prop:AP} implies that $s \in \Stab(A)$ if and only if $A$ is a union of $\frac{k \gcd(n,s)}{n} $ many full-length arithmetic progressions, each with common difference $s$. Thus, $r\in \Z$ and we obtain
\begin{equation*}
    A = \bigcup_{i=1}^{r} \left( a_i + \Stab(A) \right),
\end{equation*} 
for some $a_1,\ldots,a_r \in \Z_n$. We may assume that $a_1,\ldots,a_r \in \{0,\ldots,s-~1\}$. 
\end{proof}

As a consequence of the above result, we note that whenever $n$ and $k$ are co-prime, $\Stab(A)$ is trivial for all subsets $A \subseteq \Z_n$ of size $k$.

In the following, 
we relate the size of the stabilizer of $A$ with the number of parts in the consecutive block structure $\pi(A)$ of $A$ (Definition~\ref{cons-block}).

\begin{proposition}
Let $A \subseteq \Z_n$ with $\pi(A) = (D_1,D_2,\ldots,D_\ell)$. Then $\card{ \Stab(A)} $ divides $\ell$. \label{prop:Stab_blocks} 
\end{proposition}
\begin{proof} Let $\Stab(A) = \langle s \rangle$. If $s = 0$, then the statement follows immediately. 

 Assume $s \neq 0$ and let $r = n/s = |\Stab(A)|$. We define the following relation on the set $\{D_1,\ldots,D_\ell\}$:
 \[ D_i \sim D_j \iff D_i = D_j + t s\ \mbox{for some }\ t \in \{0,\ldots,r-1\}.\] It is easy to check that $\sim$ is an equivalence relation. We show that each equivalence class contains exactly $r$ elements. For any $i \in \{1,\ldots,\ell\}$ and $t \in \{0,\ldots,r-1\}$, we have that $D_i + t s \in \{D_1,\ldots,D_\ell\}$, because $A+ts = A$. Moreover, if $D_i + t_1s = D_i + t_2 s$ for some $t_1,t_2 \in \{0,\ldots,r-1\}$, then $t_1 = t_2$. This implies that each equivalence class contains $r$ elements. Hence, $r$ divides $\ell$.
 \end{proof}

We know from Proposition~\ref{systematic} that a cyclic $k$-matroid over $\Z_n$ contains bases of the form $B_i = \{i,i+1,\ldots,i+k-1\}$ for all $i \in \{0,\ldots,n-1\}$. Clearly, all these bases belong to the same orbit, as $B_i = B_0 + i$ for each~$i$. 
Moreover, using Lemma \ref{lem:basis_exch} we get that, for each subset $Q \subset B_0$, there exists a basis of the form $(B_0 \setminus Q) \cup P$ for some $P \subset B_0^C$. 
As an application of Proposition \ref{prop:Stab_blocks}, we obtain the following bound on the size of the stabilizer of such sets.

\begin{corollary} Let $A = (B_0 \setminus Q) \cup P \subseteq \Z_n$, where $Q \subset B_0$ and $P \subset B_0^C$ with $\card{Q} = \card{P}$. Then $\card{Stab(A)} \leq 2 \card{Q}+1$.
\label{cor:stab_basis}
\end{corollary}
\begin{proof}
Let $r$ and $s$ be the number of blocks in the consecutive block structure of $Q$ and $P$, respectively. Observe that the number of blocks in the consecutive block structure of $A$ takes one of the values in the set $\{r+s-1, r+s, r+s+1\}$. By Proposition \ref{prop:Stab_blocks}, $\card{\Stab(A)}$ divides this value, and is thus bounded above by the value $2 \card{Q}+1$, as required.
\end{proof}

\begin{example}
Let $A = (\{0,1,\ldots,k-1\} \setminus \{q\}) \cup \{p\}$ for some $q \in \{0,1,\ldots,k-1\}$ and $p \in \{k,k+1,\ldots,n-1\}$. Using Corollary \ref{cor:stab_basis}, we get that $\card{\Stab(A)} \in \{1,2,3\}$. Note that if $n$ is not a multiple of 2 or 3, then $\Stab(A)$ is trivial. We examine the case $\card{ \Stab(A)}>1$. We may assume that $2 \leq k \leq n/2$, as $\Stab(A) = \Stab(A^C)$.
\begin{enumerate}
    \item $\card{\Stab(A)} = 3$ if and only if $n=6, k=3,$ and $A = \{0,2,4\}$.
    
    In this case, the consecutive block structure of $A$ is 
    $$\pi(A) = (\{0,1,\ldots,q-~1\},\{q+1,\ldots,k-1\},\{p\}).$$
    Thus, $\Stab(A) = \langle n/3 \rangle$ if and only if the sizes of each block are equal and the shift by $n/3$ permutes them. This is possible only if $q=1, k = 3, p=4$ and $n=6$.  
    
    \item $\card{\Stab(A)} = 2$ if and only if $n$ is even, and $A = \{0,n/2\}$ or $A = \{1,n/2+1\}$. 
    
    In this case, $\card{A \cap \{0,1,\ldots,n/2-1\}} = \card{A \cap \{n/2,n/2+1,\ldots,n-1\}}$. Since $k \leq n/2$, we have that $\card{A \cap \{0,1,\ldots,n/2-1\}} \geq k-1$ and $\card{A \cap \{n/2,n/2+1,\ldots,n-1\}} \leq 1$. Therefore, $k = 2$ and hence $A = \{0,n/2\}$ or $A = \{1,n/2+1\}$.
\end{enumerate}
\end{example}

\subsubsection{Number of bases}\label{subsec:orbit-stab}

Let $n \geq 3$, $k \in \{2, \ldots, n-1\}$ and $\mM = (\Z_n,\mB)$ be a cyclic $k$-matroid. Then $\mB$ is closed under the action \eqref{gpaxnall} of $\Z_n$, and hence it is a union of some orbits of subsets of size $k$. The following two results are based on the observation that the block composition $c(s+A)$ is simply a cyclic shift of the block composition $c(A)$, hence they are equal as unordered sets. This leads to a bound on the number of distinct orbits of bases of the form $(B_0 \setminus Q) \cup P$, where $B_0 = \{0,1,\ldots, k-1\}, Q \subseteq B_0$ and $P \subseteq B_0^C$, which are contained in $\mB$.

\begin{theorem}
 There are at least $\left\lfloor \frac{k}{4} \right\rfloor +1$   
orbits of bases of the form $A = (B_0 \setminus \{q\})\cup \{p\} $, where $p \in B_0^C$. Consequently, there are at least $\frac{n}{\gcd(n,k)} \left( \left\lfloor \frac{k}{4} \right\rfloor +1 \right)$ bases in $\mB$.
\label{bound:singleton_orbits}
\end{theorem}
\begin{proof}
We first note that the block composition associated to $A =(B_0 \setminus \{q\})\cup \{p\}$ has the following possibilities.
\begin{equation*}
    c(A) = \begin{cases}
        (q, k-q-1,1) & \mbox{if}\ q\notin\{0,k-1\}, \ p \notin \{k,n-1\} \\
        (k-q-1, q+1) & \mbox{if}\ q\notin\{0,k-1\}, \ p = n-1\} \\
        (q, k-q) & \mbox{if}\ q\notin\{0,k-1\}, \ p = k \\
        (k-1,1) & \mbox{if}\ q=0, p \neq k \ \mbox{or}\ q=k-1, \ p \neq n-1 \\
        (k) & \mbox{if}\ q=0, p = k \ \mbox{or}\ q=k-1, \ p = n-1. 
     \end{cases}
\end{equation*}
Now suppose that $A =(B_0 \setminus \{q\})\cup \{p\} $ and $A^\prime = (B_0 \setminus \{q^\prime\})\cup \{p^\prime\} $ are in the same orbit, i.e. $s+ A = A^\prime$ for some $s$. Clearly, the block composition of $A^\prime$ is then a shift of $c(A)$. Thus, we have one of the following cases: \begin{enumerate}
        \item $\{q, k-q-1, 1\} = \{q^\prime, k-q^\prime-1, 1\}$, i.e. $q =q^\prime$ or $q+q^\prime = k-1$.
        \item $\{q+1, k-q-1\} = \{q^\prime+1, k-q^\prime-1\}$, i.e. $q =q^\prime$ or $q+q^\prime = k$.
        \item $\{q, k-q\} = \{q^\prime, k-q^\prime\}$, i.e. $q =q^\prime$ or $q+q^\prime = k+1$.
        \item $\{q,k-q\} = \{q^\prime+1,k-q^\prime-1\}$ or $\{q+1,k-q-1\} = \{q^\prime, k-q^\prime\}$, i.e. $q=q^\prime \pm 1$ or $q+q^\prime = k-1$.
        \item $q, q^\prime \in \{0,k-1\}$.
\end{enumerate}

Lemma \ref{lem:basis_trivial_intersect} shows that there exist $(n-k)$ bases not containing $q$. Thus, if we fix $q$ and pick $q^\prime$ so that $\card{q-q^\prime}> 1$ and $q+q^\prime < k-1$, then all the resulting bases give rise to distinct orbits. In particular, there is a distinct orbit corresponding to each of the $m+1$ values $q\in \{0,2,4, \ldots, 2\cdot m\}$, where $2m+2(m-1)<k-1\leq 2(m+1) +2m$, or $m =\left\lceil \frac{k-3}{4}\right\rceil = \left\lfloor \frac{k}{4} \right\rfloor$. 
The final result on the number of bases follows using Corollary \ref{cor:AP-union}, which implies that the size of an orbit is at least $n/\gcd(n,k)$.
\end{proof}

We further improve the lower bound on the number of orbits by considering the bases of the form $(B_0 \setminus Q) \cup P$, where $\card{Q} > 1$ and $P \subseteq B_0^C$.

\begin{theorem}
The total number of orbits in $\mB$ is bounded from below by $M+ \left\lfloor \frac{k}{4}\right\rfloor+1$, where 
 $$M= \left\lfloor\log_2 \left(\dfrac{\left\lfloor\frac{k}{2}\right\rfloor +1}{3}\right)\right\rfloor.$$ 
In particular, we have the following lower bound on the number of bases in $\mB$:   
\begin{align*}
    m_3(n,k) &=\left(M + \left\lfloor \frac{k}{4}\right\rfloor+1 \right) \frac{n}{\gcd(n,k)}. \end{align*} \label{bound:general_orbits}
\end{theorem}
\begin{proof}
Let $r$ and $s$ denote respectively the number of blocks in the consecutive block structure of $Q$ and $P$. It is easy to see that the number of blocks in the block structure of the basis element $A=(B_0\setminus Q) \cup P$ is at least $r+s-1$ and at most $r+s+1$. Now, if $\card{Q } \leq \left\lfloor k/2 \right\rfloor $, we can choose $Q$ so
that $r = \card{Q }$. Then $A$ has at least $r$ and at most $2 r + 1$ blocks in its decomposition. 

Consider the finite sequence $S=(x_1=1, x_2=4, \ldots, x_i, \ldots, x_\ell)$, where $x_i = 2 x_{i-1}+2$ for $2 \leq i\leq \ell$, and where $\ell$ is such that $x_\ell \leq \left\lfloor k/2 \right\rfloor $ and $2 x_\ell +2 \geq k/2$. Then for each pair of distinct $r,r^\prime \in S$, there are sets $Q_r$, $Q_{r^\prime}$ with $\card{Q_r } =r$, $\card{Q_{r^\prime} } =r^\prime$ that have respectively $r$ and $r^\prime$ blocks, and give rise to distinct orbits. Thus, the size $\ell$ of the sequence $S$ gives a lower bound for the number of orbits. 

We have, $x_i = 2 x_{i-1} + 2 = 2+ 2^2 +2^3+ \ldots + 2^{i-2} + 2^i = 3\cdot 2^{i-1} -2$. Moreover, for each index $i$, $3\cdot 2^{i-1}-2 = x_i \leq \left\lfloor k/2 \right\rfloor$, or $i \leq \log_2 \left(\dfrac{\left\lfloor k/2 \right\rfloor +2}{3}\right)+1$. Therefore we have
$$\ell = \left\lfloor\log_2 \left(\dfrac{\left\lfloor k/2 \right\rfloor +2}{3}\right)\right\rfloor+1.$$ From the discussion above, it is clear that the number of orbits arising from the case $\card{Q}>1$ is at least $\ell-1$. From Theorem \ref{bound:singleton_orbits} we also have the lower bound $\left\lfloor {k}/{4}\right\rfloor+1$ for the number of orbits for $\card{Q}  =1$. Plugging in this value then gives the result.
\end{proof}

\begin{remark}
Combining the results from Theorems \ref{thm:m1bound}, \ref{thm:m2bound} and \ref{bound:general_orbits}, we get three distinct lower bounds $m_1(n,k)$, $m_2(n,k)$ and $m_3(n,k)$ on the number of basis elements. We can further improve these bounds using the following two observations:
\begin{enumerate}
    \item The dual of a cyclic $k$-matroid is a cyclic $(n-k)$-matroid. Thus, $m_1(n,n-k), m_2(n,n-k)$ and $m_3(n,n-k)$ are also lower bounds on the number of bases in a cyclic $k$-matroid.
 
    \item When $\gcd(n,k) =1$, each orbit under the action \eqref{gpaxnall} has $n$ elements. Since the set of bases $\mB$ is a collection of orbits, the number of bases in $\mB$ in this case must be a multiple of $n$. So, the lower bound $m_i(n,k)$ can be improved to $\left\lceil \frac{m_i}{n}\right\rceil n$ for $i=1,2,3$.
   
\end{enumerate}
\end{remark}

From the above remark, we obtain the best possible lower bound $m_B$ from the discussed theory: 

\[m_B(n,k) = \begin{cases} 
\max \left \{ m_i(n,t) ~ \middle \vert ~  i \in \{1,2,3\}, t \in \{k,n-k\} \right \} & \mbox{if} \ \gcd(n,k) \neq 1 \\
\max \left \{ \left \lceil \dfrac{m_i(n,t)}{n} \right \rceil n  ~ \middle\vert ~  i \in \{1,2,3\}, t \in \{k,n-k\} \right \} & \mbox{if} \ \gcd(n,k) = 1.
\end{cases}\]

\subsection{Experimental Results}\label{sec:expt}
In the following discussion, we compare explicit values of the different bounds for some values of $n$ and $k$.  In Table \ref{tab:comparison_values}, we provide some values for each bound on the number of bases for different $n$ and $k$. Note that in the case of $n=6, \ k=3$, the bound 8 is achieved exactly in the cyclic matroid 1 of Example \ref{example:n6}. We graphically show in Figures \ref{fig:nfixed} and \ref{fig:kfixed} the variation of the bounds $m_1$, $m_2$ and $m_3$ with different values of, and relationships between, $n$ and $k$.

Observe that the bounds $m_1$ and $m_2$ by  far exceed $m_3$ for the  ``middle" values of $k$, whereas $m_3$ becomes significant when $k$ is large enough, particularly when $k$ attains its maximum value $k=n-2$ (We disregard the case $k=n-1$ since this only gives the uniform matroid). It is also observed that the difference between $m_2$ and the other two bounds increases rapidly as $n$ is increased. This is to be expected since $m_2$ counts more types of bases than $m_1$ and $m_3$, and the possibilities for these types grow rapidly with $n$.

\begin{table}[htb]
\begin{subtable}[h]{\textwidth}
    \centering
    \begin{tabular}{@{}|c|c|c|c|c|c|c||c|@{}}

    \toprule
    $k$ & $m_1(n,k)$ & $m_1(n,n-k)$ & $m_2(n,k) $ & $m_2(n,n-k)$ & $m_3(n,k)$ & $m_3(n,n-k)$ & $m_B$\\ \midrule
    2 & 8 & 8 & 8 & 4 & 3 & 6  & 8  \\
    3 & 8 & 8 & 8 & 8  & 2 & 2 & 8 \\ \bottomrule
    \end{tabular} 
    \caption{$n=6$}
    \label{tab:my_label1}
\end{subtable}  \\[10pt]
    
\begin{subtable}[h]{\textwidth}
\centering
\begin{tabular}{|c|c|c|c|c|c|c||c|}

\toprule
$k$ & $m_1(n,k)$ & $m_1(n,n-k)$ & $m_2(n,k) $ & $m_2(n,n-k)$ & $m_3(n,k)$ & $m_3(n,n-k)$ & $m_B$\\ \midrule
2 & 19 & 18 & 25 & 29 & 11 & 44 & 44  \\
3 & 27 & 28 & 29 & 26  & 11 & 44& 44 \\
4 & 35 & 40 & 25  & 8 & 22 & 22 & 44 \\
5 & 47 & 48 & 100 & 2 & 22 & 22 & 110  \\ \bottomrule
\end{tabular} 
    \caption{$n=11$}
    \label{tab:my_label2}
        \end{subtable} \\[10pt]
        
\begin{subtable}[h]{\textwidth}
\centering
\begin{tabular}{@{}|c|c|c|c|c|c|c||c|@{}}

\toprule
$k$ & $m_1(n,k)$ & $m_1(n,n-k)$ & $m_2(n,k) $ & $m_2(n,n-k)$ & $m_3(n,k)$ & $m_3(n,n-k)$ & $m_B$\\ \midrule
2 & 27 & 26 & 42 & 67 & 15 & 75 & 75  \\
3 & 40 & 44 & 63 & 33  & 5 & 25& 130 \\
4 & 60 & 72 & 69  & 99 & 30 & 60 & 105 \\
5 & 85 & 112 & 112 & 32 & 6 & 12 & 112 \\
6 & 124 & 160 & 223  & 8 & 10 & 20 & 223 \\
7 & 156 & 192 & 518 & 2 & 30 & 60 & 525 \\ \bottomrule
\end{tabular}  
    \caption{$n=15$}
    \label{tab:my_label3}
\end{subtable}
    
    \caption{Comparison of the lower bounds on the number of bases in an arbitrary cyclic $k$-matroid over a ground set of size $n$.}\label{tab:comparison_values}
    
    \end{table}

\begin{figure}[H]
    \centering
    \begin{subfigure}[b]{0.33\textwidth}
        \includegraphics[width=5.9cm, height=5cm]{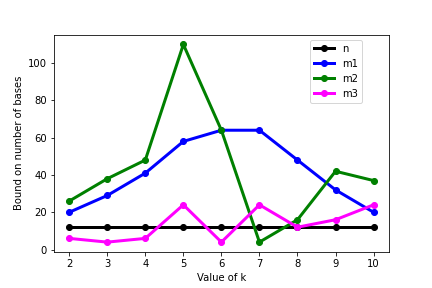}
        \caption{$n=12$}
    \end{subfigure}
    \begin{subfigure}[b]{0.32\textwidth}
        \includegraphics[width=5.9cm, height=5cm]{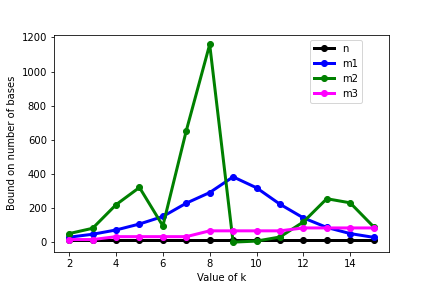}
        \caption{$n=17$}%
    \end{subfigure}
    \begin{subfigure}[b]{0.33\textwidth}
        \includegraphics[width=5.9cm, height=5cm]{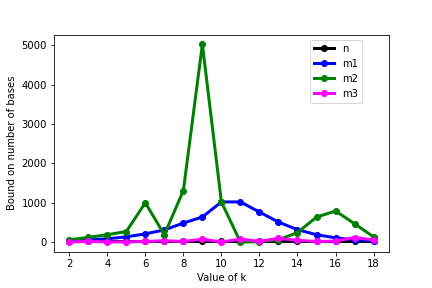}
        \caption{$n=20$}%
    \end{subfigure}
    \caption{Comparison of the bounds for fixed values of $n$ and varying $k$.}
    \label{fig:nfixed}
\end{figure}

\begin{figure}[H]
    \centering
    \begin{subfigure}[b]{0.33\textwidth}
        \includegraphics[width=5.5cm, height=4.5cm]{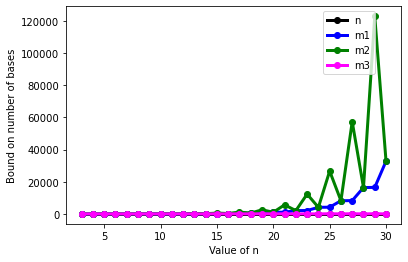}
        \caption{$n \leq 30, \ k=\lfloor\frac{n}{2}\rfloor$}%
    \end{subfigure}
    \begin{subfigure}[b]{0.32\textwidth}
        \includegraphics[width=5.9cm, height=5cm]{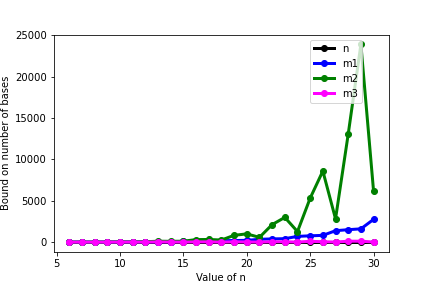}
        \caption{$n \leq 30, \ k=\lfloor\frac{n}{3}\rfloor$}%
    \end{subfigure}
    \begin{subfigure}[b]{0.33\textwidth}
        \includegraphics[width=5.9cm, height=5cm]{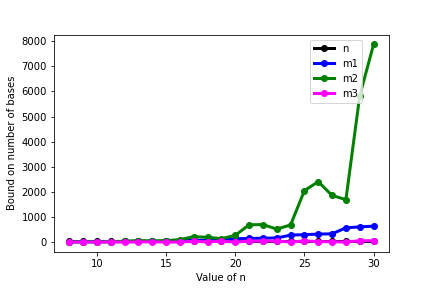}
        \caption{$n \leq 30, \ k=\lfloor\frac{n}{4}\rfloor$}%
    \end{subfigure} \\
    
        \begin{subfigure}[b]{0.33\textwidth}
        \includegraphics[width=5.9cm, height=5cm]{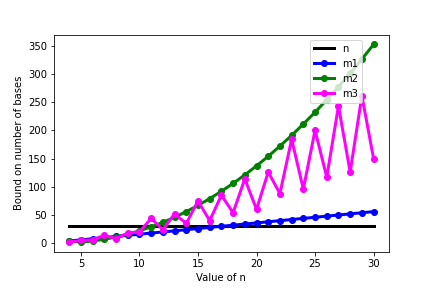}
        \caption{$n \leq 30, \ k=n-2 $}%
    \end{subfigure}
    \begin{subfigure}[b]{0.32\textwidth}
        \includegraphics[width=5.9cm, height=5cm]{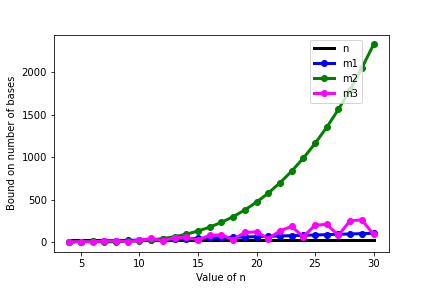}
        \caption{$n \leq 30, \ k=n-3$}%
    \end{subfigure}
    \begin{subfigure}[b]{0.33\textwidth}
        \includegraphics[width=5.9cm, height=5cm]{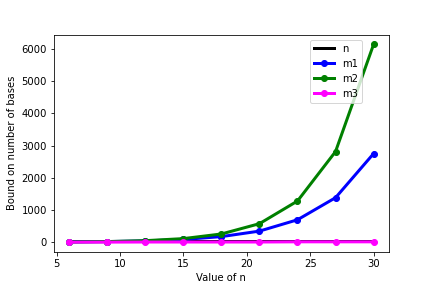}
        \caption{$n \leq 30, \ 3 \mid n, \ k=n/3$}%
    \end{subfigure} 
    \caption{Comparison of the bounds for fixed values of $k$ and varying $n$.}
    \label{fig:kfixed}
\end{figure}

In \ref{appendix:examples} we provide some explicit examples of cyclic matroids, generated through a randomized computer search.

\section{Algebraic and Geometric Connections}\label{sec:examples}

In this section, we provide examples of algebraic, geometric and combinatorial objects that may be linked to cyclic matroids. The main objects of interest are the ones introduced in Sections \ref{subsec:projplanes}, \ref{subsec:codes} and \ref{subsec:knorm}. 

\subsection{Cyclic projective planes and cyclic codes}

There are many works on the existence and non-existence of cyclic projective planes and their collineation groups. We refer the interested reader to \cite{hall1947cyclic, berman1953finite, rosati1957piani}. Moreover, cyclic projective codes have been studied in relation with designs, difference sets and cyclic codes; see \cite[Section 8.7]{huffman2010fundamentals}.

Let $q$ be a prime power and consider a Desarguesian projective plane $\PG(2,q)$ and its incidence matrix $A$. Then it is known that $A$ is necessarily a circulant matrix; see for instance \cite[Theorem 4.2.2 and its Corollary]{hirschfeld1998projective}. Since the entries of $A$ are only $0$s and $1$s, we can consider $A$ as a matrix over a finite field $\F_p$ for $p$ prime. In this case, the rank of $A$ has been completely determined by Graham and MacWilliams in \cite{graham1966number}. Here we state their result only for $p=2$.

\begin{theorem}
The rank over $\F_2$ of the incidence matrix $A$ of a Desarguesian projective plane $\PG(2,q)$  is 
$$k=\begin{cases}
q^2+q & \textnormal{ if $q$ is odd } \\
3^t+1 & \textnormal{ if $q=2^t$}.
\end{cases}$$
\end{theorem}

\begin{corollary}
Let $\mM(A)$ be the representable matroid constructed from the incidence matrix $A$ of a cyclic projective plane. Then, $\mM(A)$ is a  cyclic $k$-matroid, representable over $\F_2$, with $k$ equal to the rank of $A$.
\end{corollary}

Moreover, Pless showed also that the incidence matrix of a Desarguesian cyclic projective plane generates a binary cyclic code; see \cite{pless1986cyclic}. Hence, this class of cyclic matroids is a subclass of the one deriving from cyclic codes. 

Let $\mC$ be a $k$-dimensional linear code in $\F_q^n$. Then, using a generator matrix $G$ of $\mC$, we can associate a representable matroid $\mM_{\mC}=\mM(G)$. 

\begin{proposition}\label{prop:cyclic-codes}
Let $\mC$ be an $[n,k]_q$ cyclic code with generator matrix $G$. Then $\mM(G)$ is a cyclic $k$-matroid.
\end{proposition}
\begin{proof}
The bases in the matroid $\mM(G)$ (see Example~\ref{rep-mat} for the definition of $\mM(G)$) correspond to the sets of indices of $k$ linearly independent columns. Observe that since $\mC$ is cyclic, then the matrix $\sh(G)$ obtained by shifting on the right every row of $G$ is still a generator matrix for $\mC$. In particular, the cyclic shift of each basis is still a basis.
\end{proof}

\begin{remark}\label{rem:cyclicity}
It is immediate to see that the cyclic property is not invariant under permutation of the coordinates. Hence, in general, it is necessary to find an appropriate relabelling of the points of the matroid, in order to obtain a cyclic one. The same property it is not invariant for cyclic codes, i.e. cyclicity is not preserved under permutation of columns.
\end{remark}

\begin{remark}
Note that in general the cyclic matroid arising from cyclic codes do not satisfy the cyclicity property defined by Welsh. Indeed, for a binary cyclic code $\mC$ of odd length $n$, with $n\geq 3$, it is not difficult to see that the automorphism group of $\mC$ strictly contains $\Z_n$; see \cite{bienert2010automorphism}.
\end{remark}

\begin{example}\label{fano}
Consider the simplex code $[7,3]_2$ with generator matrix 
$$ G=\begin{bmatrix}
1 & 0 & 1 & 1 & 1 & 0 & 0 \\
0 & 1 & 0 & 1 & 1 & 1 & 0 \\
0 & 0 & 1 & 0 & 1 & 1 & 1 
\end{bmatrix}.$$
$G$ is clearly the generator matrix of a cyclic code. Moreover, the matroid $\mM(G)$ associate to it is the well-known \textbf{Fano matroid}, whose name derives from the Fano plane $\mathrm{PG}(2,2)$. This is denoted by $F_7$, the ground set is $E=\{0,1,\dots,6\}$ and the set of bases is 
\begin{align*} 
\mB=
\{&\{ 0, 3, 6 \},\{ 0, 2, 5 \},\{ 0, 2, 4 \},\{ 3, 4, 5 \},\{ 0, 2, 3 \},\{ 0, 1, 5 \},\{ 1, 2, 5 \}, \\
& \{ 2, 3, 6 \},\{ 0, 1, 4 \},\{ 0, 4, 6 \},\{ 1, 3, 5 \},\{ 2, 5, 6 \},\{ 1, 3, 6 \},\{ 0, 3, 5 \}, \\
& \{ 2, 4, 5 \},\{ 1, 2, 3 \},\{ 3, 5, 6 \},\{ 0, 1, 2 \},\{ 0, 1, 6 \},\{ 2, 3, 4 \},\{ 0, 5, 6 \}, \\ &\{ 0, 3, 4 \},\{ 1, 2, 6 \},\{ 1, 3, 4 \},\{ 1, 4, 5 \},\{ 4, 5, 6 \},\{ 2, 4, 6 \},\{ 1, 4, 6 \}\},\end{align*}
and it is not difficult to see that it satisfies the property of cyclic $3$-matroids. 
It can be graphically represented as in Figure \ref{fig:fanomatroid}, where each base is made of three points that are not collinear.
\begin{figure}[H]
\begin{center}
\begin{tikzpicture}
\tikzstyle{point}=[ball color=white, circle, draw=black, inner sep=0.1cm]
\node (v7) at (0,0) [point] {4};
\draw (0,0) circle (1cm);
\node (v1) at (90:2cm)  [point] {6};
\node (v2) at (210:2cm) [point] {0};
\node (v4) at (330:2cm) [point] {1};
\node (v3) at (150:1cm) [point] {2};
\node (v6) at (270:1cm) [point] {3};
\node (v5) at (30:1cm) [point] {5};
\draw (v1) -- (v3) -- (v2);
\draw (v2) -- (v6) -- (v4);
\draw (v4) -- (v5) -- (v1);
\draw (v3) -- (v7) -- (v4);
\draw (v5) -- (v7) -- (v2);
\draw (v6) -- (v7) -- (v1);
\end{tikzpicture}
\end{center}
    \caption{Cyclic Fano matroid $F_7$}
    \label{fig:fanomatroid}
\end{figure}
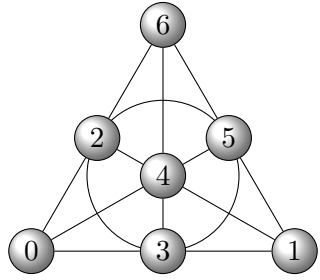

\end{example}

We do not know if in general the converse of Proposition \ref{prop:cyclic-codes} is true or not. We state this as an open problem.

\begin{problem}
Verify the converse of Proposition \ref{prop:cyclic-codes}, i.e. given a representable cyclic $k$-matroid $\mM$ on $n$ elements, determine if there exists a field $\F_q$ and an $[n,k]_q$ cyclic code $\mC$ with $\mM_\mC = \mM$.
\end{problem}

In some cases, experimental results show that the answer to the previous open problem should be positive, as the next example illustrates.

\begin{example} \label{example:n6}
Let $n=6$ and $k=3$. An exhaustive computer program shows that there are exactly three cyclic matroids on the ground set $E = \{0,1,2,3,4,5\}$ and having rank $3$. Each of the three matroids corresponds to a cyclic code. 
\begin{enumerate}
    \item $\mM_1$ with $8$ bases elements, comprising of $6$ bases from the orbit of $\{0,1,2\}$ and $2$ bases from the orbit of $\{0,2,4\}$. \\For $\mC_1 = \langle x^3 + 1 \rangle \subseteq \F_2[x]/\langle x^6-1 \rangle$, we get $\mM_1 = \mM_{\mC_1}$.
    \item $\mM_2$ with $18$ bases elements, comprising of $6$ bases from the orbit of $\{0,1,2\}$, $6$ bases from the orbit of $\{0,1,3\}$, and $6$ bases from the orbit of $\{0,1,4\}$. \\For $\mC_2 = \langle x^3 + 2x^2+2x+ 1 \rangle \subseteq \F_3[x]/\langle x^6-1 \rangle$, we get $\mM_2 = \mM_{\mC_2}$.
    \item $\mM_3$ with $20$ bases elements, comprising of all the $20$ subsets of size $3$. \\For $\mC_3 = \langle x^3 + 2x^2 + 2x + 1 \rangle \subseteq \F_5[x]/\langle x^6-1 \rangle$, we get $\mM_3 = \mM_{\mC_3}$.
\end{enumerate}
\end{example}


\subsection{$k$-Normal Elements}

In this subsection we establish a connection between $(n-k)$-normal elements and cyclic $k$-matroids. The connection between matroids and $k$-normal elements has never been observed before, to the best of our knowledge.

Given an $(n-k)$-normal element $\a \in  \fqn$,  let
$$V:=\mathrm{span}_\fq \left\{\alpha, \alpha^q, \ldots, \alpha^{q^{n-1}} \right\} $$ 
be the $k$-dimensional span over $\fq$ of the conjugates of $\a$. We may associate a matroid $\mM = (\Zn, \mB)$ on $n$ symbols to $\a$ as follows.
Let $\mI$ be a collection of subsets of $\Z_n$ such that $S\in\mI$ if and only if the set of powers $\{\a^{q^i} \mid i \in S\}$ is linearly independent over $\fq$. Then $\mI$ clearly satisfy the axioms (I1)--(I3), hence it is the collection of independent sets of a matroid. In particular, the collection of bases of such a matroid is defined as 
$$\mB = \{\{i_1, \ldots, i_k\} \subseteq \Z_n \mid \{\a^{q^{i_1}}, \a^{q^{i_2}}, \ldots, \a^{q^{i_k}}\} \text{ is a linear basis of } V  \text{ as vector space over } \fq\}.$$
\begin{proposition}
The matroid $\mM= (\Zn, \mB)$ associated to an $(n-k)$-normal element $\a\in\F_q^n$ is a cyclic $k$-matroid.
\end{proposition}
\begin{proof}
Since the first $k$ powers of $\a$ must be linearly independent in order for all of them to span a $k$-dimensional vector space, we must have $\{0,1, \ldots, k-1\} \in \mB$. Further, for any $s \in \Z_n$,  $\{\a^{q^{i_1}}, \a^{q^{i_2}}, \ldots, \a^{q^{i_k}}\}$ is linearly independent if and only if $\{\a^{q^{i_1+s}}, \a^{q^{i_2+s}}, \ldots, \a^{q^{i_k+s}}\}$ is linearly independent, by the properties of the Frobenius automorphism. Thus, $\mM$ is a cyclic $k$-matroid.
\end{proof}

In \cite{tinani_rosenthal_2021}, it was left as an open problem to determine which subsets of $\{\alpha, \alpha^q, \alpha^{q^2}, \ldots, \alpha^{q^{n-1}}\}$ of size $k$ or smaller, apart from $\{\alpha, \alpha^q, \alpha^{q^2}, \ldots, \alpha^{q^{k-1}}\}$,  are linearly independent, for  an  $(n-k)$-normal element  $\alpha$ of $\F_{q^n}$. 

Clearly, the results of Section \ref{number_bases} give lower bounds on the number of $k$-subsets of $\{\a, \a^q, \ldots, \a^{q^{n-1}}\}$ (where $\a$ is $(n-k)$-normal) which form bases over $\fq$, or equivalently, upper bounds on the number of dependent $k$-subsets. 

We assert that our association of $k$-normal elements with matroids strongly suggests that a complete and general solution to the mentioned problem may be very difficult to arrive at.

We further state the following open problem, whose solution we conjecture is positive, based on multiple computer experiments and observations.

\begin{problem}
Given a representable cyclic $k$-matroid $\mM = (\Zn,\mB)$, determine if there exists a prime power $q$ and an $(n-k)$-normal element $\a \in \fqn$ such that all the bases of the $\F_q$-span of $\a, a^q, \ldots, \a^{q^{n-1}}$, are given by the sets $\{\a^{q^{i_1}}, \a^{q^{i_2}}, \ldots, \a^{q^{i_k}}\}$ where $\{i_1, i_2, \ldots, i_k\} \in \mB$.
\end{problem}

\section{Orbits of the Cyclic Group Action}\label{sec:numberoforbits}

In this section, we study the orbits of the action of $\Z_n$ on $2^{\Z_n}$ from a general point of view. We first provide an algebraic characterization of orbits and then count the total number of orbits.

\subsection{Characterization of orbits}
We introduce here circulant matrices in order to characterize the orbits of $\varphi$. Given a subset $A \subseteq \Z_n$, we associate to $A$ the circulant matrix $$M_A = \left( m_{i,j} \right)_{i,j =0}^{n-1}, \ \mbox{where} \ m_{i,j} = \begin{cases} 1 & \mbox{if} \ (j-i) \bmod n \in A \\ 0 & \mbox{otherwise}. \end{cases}$$  
Note that the first row of $M_A$ is the incidence vector of $A$, i.e., its $j$-th entry is 1 if and only if $j \in A$. Similarly, the $i$-th row of $M_A$ is the incidence vector of $A+i$. Therefore, the orbit of $A$ has a one-to-one correspondence with the rows of $M_A$. 

An equivalent description can be obtained using the polynomials in the ring $\Z_2[x]/(x^n-1)$. Let $A \subseteq \Z_n$ and $M_A$ be the associated circulant matrix. Further, let $(c_0,c_1,\ldots,c_{n-1})$ be the first row of $M_A$. Then we associate to $A$ the polynomial $$f_A(x) = c_0 + c_1 x + \cdots + c_{n-1} x^{n-1} \in \Z_2[x]/(x^n-1).$$

It is easy to see that $f_{A+i}(x) = x^i f_A(x) \bmod (x^n-1)$ for all $i \in \Z_n$. 

\begin{lemma}
Let $A \subseteq \Z_n$ and $\Stab(A) = \langle s \rangle$ with $s \neq 0$. Then, 
\begin{enumerate}
    \item $\rk(M_A) \leq s$.
    \item $f_A(x) (x^s-1) = 0 \bmod (x^n-1)$.
\end{enumerate}
\end{lemma}
\begin{proof}
 \begin{enumerate}
     \item This follows from the fact that the first $s$ rows of $M_A$ repeat $|\Stab(A)|$ times. In particular, $M_A = (m_{i,j})_{i,j=0}^{n-1}$ has the following form 
      \begin{equation}
         M_A = \begin{pmatrix} C_A & \cdots & C_A \\ \vdots & \ddots & \vdots \\ C_A & \cdots & C_A \end{pmatrix},
     \end{equation} where $C_A = (m_{i,j})_{i,j=0}^{s-1}$. Thus, $\rk(M_A) = \rk(C_A) \leq s$.
     \item Since $A+s = A$, it follows that $f_A(x) = x^s f_A(x)$. Therefore, $$f_A(x) (x^s-1) = 0 \bmod (x^n-1).$$
 \end{enumerate}
\end{proof}

\begin{remark}
If $\Stab(A) = \langle s \rangle$ with $s \neq 0$, then it is not necessary that $\rk(M_A) = s$. For example, consider $A = \{0,2,3,5\} \subseteq \Z_6$. Then, $\Stab(A) = \{0,3\}$ but $\rk(M_A) = 2$.
\end{remark}

\subsection{Number of Orbits}
We now derive the total number of orbits under the action \eqref{gpaxnall}. The results obtained here are based on the observation that the block composition of the elements belonging to the same orbit have equal length and they are cyclic shift of each other.

Let us denote the set of all orbits by $\mathcal{T}$, and the set of orbits corresponding to block composition of length $r$ by $\mathcal{T}_r$. Clearly, $\mathcal{T} = \bigsqcup\limits_{r=1}^{\lfloor\frac{k}{2}\rfloor} T_r$.

Let $\mC_{m,r}$ denote the set of compositions of $m$ with length $r$. Observe that the group $\Z_r$ acts on the set $\mathcal{C}_{m,r}$ via cyclic shifts. More precisely, for $s \in \Z_r$ and $(c_1, c_2, \ldots, c_r) \in \mathcal{C}_{m,r}$, we define $s \cdot (c_1, c_2, \ldots, c_r) = (c_{s+1}, c_{s+2}, \ldots, c_r, c_1, \ldots, c_s)$.

We consider the set $\mathcal{Z}_{n,k,r} := \mathcal{C}_{k,r} \times \mathcal{C}_{n-k,r}$. The actions of $\Z_r$ on $\mathcal{C}_{k,r}$ and $\mathcal{C}_{n-k,r}$ extends in a natural way to $\mathcal{Z}_{n,k,r}$. Moreover, tuples $(c_A , c_B)$ and $(c_A^\prime , c_B^\prime)$ lie in the same orbit under this action if and only if there exists a shift $s \in \Z_r$ such that $s \cdot c_A = c_A^\prime$ and $s \cdot c_B = c_B^\prime$. Denote by $\mathcal{Z}_{n,k,r}/\Z_r$ the set of orbits under this action.

\begin{lemma} The following holds:
\begin{align*} \card{\mathcal{Z}_{n,k,r}/\Z_r} = & \frac{1}{r} \binom{k-1}{r-1} \binom{{n-k}-1}{r-1} +\\ &    \sum\limits_{s \mid \gcd(k, n-k, r)}\left(\frac{r-s}{rs}\right) \sum\limits_{s \mid \gcd(k, n-k, r)} \frac{1}{s} \binom{k/s-1}{r/s-1} \binom{{(n-k)}/s-1}{r/s-1}. \end{align*}

\end{lemma}
\begin{proof}
Let  $(c_A, c_B)$ be a tuple with $c_A=(a_1, a_2, \ldots, a_r)$, $ c_B = (b_1, b_2, \ldots, b_r)$. The orbit of $(c_A, c_B)$ is of size strictly smaller than $r$ if and only if there exists an $s$ dividing $r$ such that $a_{i} = a_{(i+s)} \mod r $ and $b_{i} = b_{(i+s)} \mod r $, for each $1\leq i \leq r$. This $s$, in fact, generates the stabilizer of the element $(c_A, c_B)$ and gives the number of repeating blocks in the compositions $c_A$ and $c_B$, or equivalently, the size of the orbit of $(c_A, c_B)$. Thus, we also have $s \mid k$ and $s \mid n-k$, and thus $s \mid \gcd(r, k, n-k)$. Each (minimal) repeating block in $c_A$ (resp. $c_B$) is a composition of $k/s$ (resp. $(n-k)/s$) with length $r/s$. Conversely, given $s \mid \gcd(r, k, n-k)$ and any composition $\tilde{c_A} = (a_1, \ldots, a_{r/s})$ of $k/s$ and $\tilde{c_B}=(b_1, \ldots, b_{r/s})$ of $n-k/s$ with length $r/s$, define the tuple $(c_A,c_B)\in \mathcal{Z}_{n,k,r}$, where
\begin{align*}
    c_A&=(a_1, \ldots, a_{r/s}, a_1, \ldots, a_{r/s}, \ldots, a_1,\ldots, a_{r/s}), \\
    c_B&=(b_1, \ldots, b_{r/s}, b_1, \ldots, b_{r/s}, \ldots, b_1,\ldots, b_{r/s}).
\end{align*}
This orbit has size $s$. Thus, the number of elements of $\mathcal{Z}_{n,k,r}$ occupying orbits of size strictly smaller than $r$ is given by $$x_{n,k,r} =\sum\limits_{s \mid \gcd(k, n-k, r)} {\card{\mathcal{C}_{k/s,r/s}\times \mathcal{C}_{n-k/s,r/s}}},$$ 
and the remaining $\card{\mathcal{Z}_{n,k,r}}-x_{n,k,r}$ elements have orbits of full size $r$. We may thus calculate the total number of orbits as \begin{align*}  \card{\mathcal{Z}_{n,k,r}/\Z_r}  &=   \dfrac{ \card{\mathcal{Z}_{n,k,r}}-x_{n,k,r}}{r} + \sum\limits_{s \mid \gcd(k, n-k, r)} \dfrac{\card{\mathcal{C}_{k/s,r/s}\times \mathcal{C}_{n-k/s,r/s}}}{s}  \\  &=  \frac{1}{r} \left[\binom{k-1}{r-1} \binom{{n-k}-1}{r-1}-\sum\limits_{s \mid \gcd(k, n-k, r)} { \binom{k/s-1}{r/s-1} \binom{{(n-k)}/s-1}{r/s-1}} \right]  \\ &+ \sum\limits_{s \mid \gcd(k, n-k, r)} \frac{1}{s} \binom{k/s-1}{r/s-1} \binom{{(n-k)}/s-1}{r/s-1} \\ &=\frac{1}{r} \binom{k-1}{r-1} \binom{{n-k}-1}{r-1} \\&+  \sum\limits_{s \mid \gcd(k, n-k, r)}\left(\frac{r-s}{rs}\right) \sum\limits_{s \mid \gcd(k, n-k, r)} \frac{1}{s} \binom{k/s-1}{r/s-1} \binom{{(n-k)}/s-1}{r/s-1}. \end{align*}
\end{proof}

\begin{theorem}
There is a bijection between the set $\mathcal{T}_r$ and $\mathcal{Z}_{n,k,r}/\Z_r$.
\end{theorem}
\begin{proof}
We define a map $\psi: \mathcal{T}_r \rightarrow \mathcal{Z}_{n,k,r}/\Z_r$ as follows. Consider an element $D \in \mathcal{T}_r$ with consecutive block structure $\pi(D)=(D_1, D_2, \ldots, D_r)$. Write $D_i=\{d_i, d_i+1, \ldots, d_i+\card{D_i}-1\}$ for $2 \leq i \leq r $, $\sum\limits_{i=1}^r \card{D_i} = k$. We have $d_1 <d_2 < \ldots < d_r$.  

We define $\psi(D)$ as the orbit defined by $(c(D),c(D^C))$  in $\mathcal{Z}_{n,k, r}/\Z_r$.

$\psi$ is well-defined:

To see this, let $l \in \Z_n$ be a shift and write $D^\prime  = D+l$. We need to show that $\psi(D) = \psi(D^\prime)$. We have $D^\prime =(D_1+l) \cup (D_2+l) \cup \ldots (D_r+l) \in \mathcal{T}$, with $D_i+l=\{d_i+l, d_i+l+1, \ldots, d_i+l+\card{D_i}-1\}$ for $2 \leq i \leq r $, $\sum\limits_{i=1}^r \card{D_i+l} = k$. Now let $1\leq s \leq r$ be the largest index such that $d_s+l \leq n$. Then since $l<n$, we have $d_j+l-n < d_1+l$ for every $s <j \leq r$, and the arrangement $D^\prime = (D_{s+1} +l) \cup (D_{s+2}+l) \cup (D_{r}+l)\cup (D_1+l)\cup \ldots \cup (D_{s}+l)$ is such that the first elements of the blocks are in ascending order. Further, $D^C = \{d_{s+1}+ c_{s+1}+l-n, \ldots, d_{s+2}+l-n-1\} \cup \ldots \cup \{d_{r}+ c_r+l-n, \ldots, d_1+l-1\}\cup \ldots \cup  \{d_{s}+ c_s+l, \ldots, d_{s+1}+l-1\}  $. Thus, by the definition of $\psi$ we have $\psi(D^\prime) = (\bar{c}, \bar{c}^\prime)$ where $\bar{c} = (c_{s+1}, \ldots, c_{r}, c_1, c_2, \ldots, c_s)$ and $\bar{c}^\prime = (c_{s+1}^\prime , \ldots, c_{r}^\prime, c_1^\prime, c_2^\prime, \ldots, c_s^\prime)$. Now, it is clear that $\psi(\bar{D}) = s \cdot \psi (D)$, and they both thus define the same orbit in $\mathcal{Z}_{n,k,r}$.

$\psi$ is injective: Let $\psi(D) = (c,c^\prime)$, with $c= (c_1, c_2, \ldots, c_r)$, $c= (c_1^\prime, c_2^\prime, \ldots, c_r^\prime)$. Since $\psi(\bar{D})$ and $\psi(D)$ represent the same orbit, the tuple corresponding to $\bar{D}$ is of the form $(s\cdot c, s \cdot c^\prime)$ for some $s \in \Z_r$. Denote by $d_1$ and $\bar{d_1}$ the smallest elements of $D$ and $\bar{D}$ respectively. Then, using the definition of $\psi$ we can reconstruct $D$ and $\bar{D}$:
\[D = \{d_1, \ldots, d_1+c_1-1\}\cup \ldots  \cup  \{d_1+\sum\limits_{i=1}^{r-1} (c_i+c_i^\prime) \ldots, d_1+\sum\limits_{i=1}^{r-1} (c_i+c_i^\prime)+c_r-1\}\]
where the $j^{th}$ block is given by $$D_j=\{d_1+\sum\limits_{i=1}^{j-1} (c_i+c_i^\prime), \ldots, d_1+\sum\limits_{i=1}^{j-1} (c_i+c_i^\prime)+c_j-1\}.$$
We have,
\begin{align*} \bar{D} = & \{ \bar{d}_1, \ldots, d_1+c_{s+1}-1\}\cup \ldots  \cup \\ &\{\bar{d}_1+\sum\limits_{i=s+1}^{r-1} (c_i+c_i^\prime) \ldots, d_1+\sum\limits_{i=s+1}^{r-1} (c_i+c_i^\prime)+c_r-1\} \cup \\ &\{\bar{d}_1 + \sum\limits_{i=s+1}^{r} (c_i+c_i^\prime), \ldots, \bar{d}_1 + \sum\limits_{i=s+1}^{r} (c_i+c_i^\prime)+c_1-1\} \cup \ldots   \\ & \cup \{\bar{d}_1 +  \sum\limits_{i=s+1}^{r} (c_i+c_i^\prime) + \sum\limits_{i=1}^{s-1} (c_i+c_i^\prime) ,\ldots, \bar{d}_1 +  \sum\limits_{i=1}^{r} (c_i+c_i^\prime) -1 \}.\end{align*}

Finally, using the fact that $\sum\limits_{i=1}^{r} (c_i+c_i^\prime) = n$, we have \begin{align*} \bar{D} +d_1-\bar{d_1}-\sum\limits_{i=s+1}^{r} (c_i+c_i^\prime) = & \{d_1+\sum\limits_{i=1}^{s} (c_i+c_i^\prime) \ldots, d_1+\sum\limits_{i=1}^{s} (c_i+c_i^\prime)+c_{s+1}-1\} \cup \ldots \\ & \cup \{d_1 + \sum\limits_{i=1}^{r-1} (c_i + c_i^\prime), \ldots, d_1+\sum\limits_{i=1}^{r-1}(c_i + c_i^\prime)+c_r-1\} \\ & \cup \{d_1, \ldots, d_1+c_1-1\}\cup \ldots \\ &  \cup  \{d_1+\sum\limits_{i=1}^{r-1} (c_i+c_i^\prime) \ldots, d_1+\sum\limits_{i=1}^{r-1} (c_i+c_i^\prime)+c_r-1\}  \\  = & D_{s+1} \cup \ldots D_r \cup D_1 \cup \ldots \cup D_s = D. \end{align*}

Thus, $D$ and $\bar{D}$ are shifts of each other, i.e. lie in the same orbit, and thus $\psi$ is injective.

$\psi$ is surjective: Given $(c,c^\prime) \in  \mathcal{Z}_{n,k,r}$, write, as before, $c=(c_1, c_2 ,\ldots, c_r)$ and $c^\prime=(c_1^\prime, c_2^\prime ,\ldots, c_r^\prime) $ define \begin{align*}\label{preimage} \small{ D :=\{0, 1, \ldots, c_1-1 \} \cup \{c_1+c_1^\prime, \ldots, c_1+c_1^\prime+c_2-1\} \cup \ldots \cup \left\{\sum\limits_{i=1}^{r-1} (c_i+c_i^\prime), \ldots, n-c_r^\prime-1  \right\} \in \mathcal{T}.}\end{align*} Note that this can always be done since $r \leq \lfloor \frac{k}{2} \rfloor $. We have, $D^C = \{c_1, c_1+1, \ldots, c_1+c_1^\prime-1\}\cup \{ c_1+c_1^\prime+c_2^\prime,  \ldots, c_1+c_1^\prime+c_2+c_2^\prime-1\} \cup \ldots \cup \left\{\sum\limits_{i=1}^{r-1} (c_i+c_i^\prime)+c_{r} , \ldots, n-1\right\}$. It is now clear that $\psi(D) = (c,c^\prime)$. This completes the proof.
\end{proof}

\begin{corollary} There is a bijection between the set $\mathcal{T}$ and $\bigsqcup\limits_{r=1}^{\lfloor\frac{k}{2}\rfloor}\mathcal{Z}_{n,k}$. Hence, the number of orbits under the action \ref{gpaxnall} is given by
\begin{align*}\card{\mathcal{Z}_{n,k} } = & \sum\limits_{1\leq r \leq \lfloor \frac{k}{2} \rfloor} \ \left[\sum\limits_{s \mid \gcd(k, n-k, r)} \dfrac{\card{\mathcal{C}_{k/s,r/s}\times \mathcal{C}_{n-k/s,r/s}}}{s} + \sum\limits_{s \nmid \gcd(k,n-k, r)} \dfrac{ \card{\mathcal{C}_{k,r} \times \mathcal{C}_{n-k,r}}}{r} \right]\\ = & \sum\limits_{1\leq r \leq \lfloor \frac{k}{2} \rfloor} \ \left[ \sum\limits_{s \mid \gcd(k, n-k, r)} \frac{1}{s} \binom{k/s-1}{r/s-1} \binom{{(n-k)}/s-1}{r/s-1} + \sum\limits_{s \nmid \gcd(k,n-k, r)} \frac{1}{r} \binom{k-1}{r-1} \binom{{n-k}-1}{r-1}  \right].\end{align*}
\end{corollary}

\bigskip
\section*{Acknowledgements}

The work of Gianira N. Alfarano is supported by the Swiss National Science Foundation under grant n. 188430. The work of Simran Tinani is supported by armasuisse Science and Technology. The work of Karan Khathuria is supported by the Estonian Research Council grant PRG49.

\bibliographystyle{abbrv} 
 \bibliography{references.bib}

\begin{thebibliography}{10}

\bibitem{alon1989combinatorial}
N.~Alon, Y.~Caro, I.~Krasikov, and Y.~Roditty.
\newblock Combinatorial reconstruction problems.
\newblock {\em Journal of Combinatorial Theory, Series B}, 47(2):153--161,
  1989.

\bibitem{berman1953finite}
G.~Berman.
\newblock Finite projective plane geometries and difference sets.
\newblock {\em Transactions of the American Mathematical Society},
  74(3):492--499, 1953.

\bibitem{bienert2010automorphism}
R.~Bienert and B.~Klopsch.
\newblock Automorphism groups of cyclic codes.
\newblock {\em Journal of Algebraic Combinatorics}, 31(1):33--52, 2010.

\bibitem{colbourn1995complexity}
C.~J. Colbourn, J.~S. Provan, and D.~Vertigan.
\newblock The complexity of computing the tutte polynomial on transversal
  matroids.
\newblock {\em Combinatorica}, 15(1):1--10, 1995.

\bibitem{gimenez2005number}
O.~Gim{\'e}nez, A.~De~Mier, and M.~Noy.
\newblock On the number of bases of bicircular matroids.
\newblock {\em Annals of Combinatorics}, 9(1):35--45, 2005.

\bibitem{gimenez2006complexity}
O.~Gim{\'e}nez and M.~Noy.
\newblock On the complexity of computing the tutte polynomial of bicircular
  matroids.
\newblock {\em Combinatorics, Probability and Computing}, 15(3):385--395, 2006.

\bibitem{graham1966number}
R.~Graham and J.~MacWilliams.
\newblock On the number of information symbols in difference-set cyclic codes.
\newblock {\em Bell System Technical Journal}, 45(7):1057--1070, 1966.

\bibitem{greene1973multiple}
C.~Greene.
\newblock A multiple exchange property for bases.
\newblock {\em Proceedings of the American Mathematical Society}, 39(1):45--50,
  1973.

\bibitem{greene}
C.~Greene.
\newblock Weight enumeration and the geometry of linear codes.
\newblock {\em Studies in Applied Mathematics}, 55(2):119--128, 1976.

\bibitem{guo2021approximately}
H.~Guo and M.~Jerrum.
\newblock Approximately counting bases of bicircular matroids.
\newblock {\em Combinatorics, Probability and Computing}, 30(1):124--135, 2021.

\bibitem{hall1947cyclic}
M.~Hall~Jr.
\newblock Cyclic projective planes.
\newblock {\em Duke Mathematical Journal}, 14(4):1079--1090, 1947.

\bibitem{hirschfeld1998projective}
J.~Hirschfeld.
\newblock {\em Projective geometries over finite fields. Oxford Mathematical
  Monographs}.
\newblock Oxford University Press New York, 1998.

\bibitem{huczynska2013existence}
S.~Huczynska, G.~L. Mullen, D.~Panario, and D.~Thomson.
\newblock Existence and properties of {$k$}-normal elements over finite fields.
\newblock {\em Finite Fields and their Applications}, 24:170--183, 2013.

\bibitem{huffman2010fundamentals}
W.~C. Huffman and V.~Pless.
\newblock {\em Fundamentals of error-correcting codes}.
\newblock Cambridge university press, 2010.

\bibitem{lomeli1996randomized}
L.~C. Lomel{\'\i} and D.~Welsh.
\newblock Randomized approximation of the number of bases.
\newblock {\em Contemporary Mathematics}, 197:371--376, 1996.

\bibitem{mnukhin1992k}
V.~B. Mnukhin.
\newblock The $k$-orbit reconstruction and the orbit algebra.
\newblock {\em Acta Applicandae Mathematica}, 29(1-2):83--117, 1992.

\bibitem{oxley2006matroid}
J.~G. Oxley.
\newblock {\em Matroid theory}, volume~3.
\newblock Oxford University Press, USA, 2006.

\bibitem{pendavingh2018number}
R.~Pendavingh and J.~Van Der~Pol.
\newblock On the number of bases of almost all matroids.
\newblock {\em Combinatorica}, 38(4):955--985, 2018.

\bibitem{pless1986cyclic}
V.~Pless.
\newblock Cyclic projective planes and binary, extended cyclic self-dual codes.
\newblock {\em Journal of Combinatorial Theory, Series A}, 43(2):331--333,
  1986.

\bibitem{radcliffe2006reconstructing}
A.~J. Radcliffe and A.~D. Scott.
\newblock Reconstructing under group actions.
\newblock {\em Graphs and Combinatorics}, 22(3):399--419, 2006.

\bibitem{reis2019existence}
L.~Reis.
\newblock Existence results on $ k $-normal elements over finite fields.
\newblock {\em Revista Matem{\'a}tica Iberoamericana}, 35(3):805--822, 2019.

\bibitem{rosati1957piani}
L.~A. Rosati.
\newblock Piani proiettivi desarguesiani non ciclici.
\newblock {\em Bollettino dell'Unione Matematica Italiana}, 12(2):230--240,
  1957.

\bibitem{simon2018combinatorial}
J.~Simon.
\newblock The combinatorial $k$-deck.
\newblock {\em Graphs and Combinatorics}, 34(6):1597--1618, 2018.

\bibitem{snook2012counting}
M.~Snook.
\newblock Counting bases of representable matroids.
\newblock {\em The electronic journal of combinatorics}, pages P41--P41, 2012.

\bibitem{tinani_rosenthal_2021}
S.~Tinani and J.~Rosenthal.
\newblock Existence and cardinality of $k$-normal elements in finite fields.
\newblock In {\em Arithmetic of Finite Fields. 8th International Workshop,
  WAIFI 2020, Rennes, France}, Theoretical Computer Science and General Issues.
  Springer International Publishing, 2021.

\bibitem{van2012introduction}
J.~H. Van~Lint.
\newblock {\em Introduction to coding theory}.
\newblock Graduate Texts in Mathematics. Springer, Berlin, Heidelberg, 1999.

\bibitem{welsh2010matroid}
D.~J. Welsh.
\newblock {\em Matroid theory}.
\newblock Dover Publications Inc., 2010.

\end{thebibliography}
\newpage
\appendix

\section{Explicit Cyclic $k$-Matroids} \label{appendix:examples}

Computer search was used to find examples of cyclic $k$-matroids different from the uniform matroid. Due to the randomized nature of the algorithm used, the matroids obtained were all quite close to the uniform matroid, usually missing one or two cyclic orbits. We list the orbit representatives of the basis set for a few select cases. Note that for given $n$ and $k$, the listed matroid need not be the only non-uniform cyclic matroid. The working code for the generation algorithm, as well as for the bound calculations in Section \ref{sec:expt} can be found at \url{https://github.com/simran-tinani/Cyclic-matroids}.
\\

\begin{table}[H]
 \begin{tabular}{@{}|c|c|c|@{}}

\toprule
$n$ & $k$ & Basis Orbit Representatives \\ \midrule

6 & 3 & $\{0,1,2\}, \{0, 2, 4\}$  \\ \hline

6 & 4 & $\{0, 1, 2, 4\}, \{0, 1, 2, 3\}$ \\ \hline

7 & 3 & $\{0, 1, 4\}, \{0, 1, 2\}, \{1, 3, 6\}, \{1, 5, 6\}$ \\ \hline

9& 3 & $\{2, 5, 7\},
 \{0, 1, 2\},
 \{0, 3, 8\},
 \{1, 7, 8\},
 \{1, 4, 5\},
 \{0, 6, 8\},
 \{0, 5, 7\},
 \{2, 4, 7\},
 \{1, 5, 6\}$\\ \hline

9 & 4 & \begin{tabular}{@{}c@{}}$\{3, 4, 6, 8\},
 \{1, 3, 4, 5\}, \{0, 3, 5, 6\}, \{0, 5, 6, 8\}, \{1, 2, 3, 6\},  \{0, 1, 2, 6\}$ \\ $\{0, 2, 5, 8\}, \{0, 1, 2, 3\}, \{2, 3, 6, 7\},  \{2, 4, 6, 7\}, \{0, 3, 5, 7\}, \{3, 4, 5, 7\}, \{1, 5, 7, 8\}$ \end{tabular} \\ \hline

 10 & 6 & \begin{tabular}{@{}c@{}}
     $\{0, 1, 2, 4, 7, 9\}, \{2, 3, 6, 7, 8, 9\}, \{1, 2, 3, 5, 6, 7\}, \{0, 1, 2, 4, 8, 9\}$\\ $\{1, 3, 4, 5, 6, 9\}, \{0, 1, 3, 4, 6, 8\},\{2, 3, 4, 7, 8, 9\},
  \{2, 3, 4, 6, 7, 9\}$\\ $\{1, 2, 3, 4, 5, 8\}, \{0, 2, 3, 6, 7, 8\}, \{0, 2, 5, 6, 8, 9\}, \{1, 2, 4, 5, 6, 7\},$\\ $\{2, 3, 4, 5, 7, 8\},
  \{0, 1, 2, 3, 4, 5\}, \{0, 1, 3, 5, 8, 9\}, \{0, 1, 2, 3, 7, 9\},$\\ $\{1, 3, 4, 6, 8, 9\}, \{0, 1, 3, 5, 6, 7\}, \{1, 2, 3, 6, 8, 9\},$\\ $\{1, 2, 3, 5, 6, 9\}, \{0, 2, 3, 6, 7, 9\}, \{0, 1, 2, 4, 6, 8\}$
  \end{tabular} \\ \hline

11 & 4 & \begin{tabular}{@{}c@{}}$\{0, 3, 5, 10\},
 \{3, 5, 7, 8\},
 \{0, 1, 5, 9\},
 \{1, 4, 6, 9\},
 \{2, 6, 9, 10\},
 \{1, 6, 7, 9\}$\\
 $\{0, 5, 8, 9\},
 \{0, 5, 7, 10\},
 \{0, 3, 4, 8\},
 \{0, 2, 8, 9\},
 \{5, 6, 7, 10\},
 \{3, 5, 7, 10\}$\\
 $\{0, 4, 5, 9\},
 \{2, 6, 7, 9\},
 \{0, 1, 2, 9\},
 \{1, 3, 8, 10\},
 \{0, 4, 5, 6\},
 \{2, 3, 8, 9\}$\\
 $\{0, 1, 2, 6\},
 \{1, 2, 7, 10\},
 \{0, 1, 2, 3\},
 \{2, 3, 6, 9\},
 \{2, 3, 6, 7\},
 \{0, 2, 3, 10\}$\\
 $\{2, 4, 5, 8\},
 \{0, 1, 3, 9\},
 \{1, 2, 3, 9\},
 \{0, 1, 2, 4\},
 \{2, 4, 8, 10\}$
\end{tabular} \\ \hline
 
13& 3 & \begin{tabular}{@{}c@{}} $\{3, 4, 7\},
 \{0, 1, 2\},
 \{4, 11, 12\},
 \{0, 4, 12\},
 \{7, 10, 12\},
 \{0, 6, 12\}$\\
 $\{3, 7, 9\},
 \{4, 8, 11\},
 \{2, 3, 10\},
 \{5, 7, 8\},
 \{0, 7, 9\},
 \{6, 8, 11\}$\\
 $\{0, 6, 8\},
 \{3, 7, 12\},
 \{2, 5, 6\},
 \{4, 6, 12\},
 \{7, 8, 10\},
 \{7, 9, 11\}$\\
 $\{1, 6, 11\},
 \{1, 5, 6\},
 \{1, 8, 11\}$ 
 \end{tabular} \\ \hline
 
15 & 3 & \begin{tabular}{@{}c@{}} $\{0, 1, 2\},
 \{3, 12, 13\},
 \{0, 4, 12\},
 \{5, 13, 14\},
 \{9, 10, 13\},
 \{7, 10, 13\},$\\
 $\{3, 7, 9\},
 \{4, 8, 11\},
 \{6, 11, 14\},
 \{1, 5, 12\},
 \{4, 9, 10\},
 \{3, 5, 13\},$\\
 $\{0, 2, 7\},
 \{5, 6, 14\},
 \{2, 3, 10\},
 \{5, 7, 8\},
 \{0, 7, 9\},
 \{6, 8, 11\}$,\\
 $\{0, 1, 12\},
 \{2, 6, 12\},
 \{0, 6, 8\}
 \{3, 9, 12\},
 \{7, 9, 11\},
 \{3, 13, 14\}$,\\
 $\{7, 8, 10\},
 \{2, 11, 13\},
 \{2, 4, 14\},
 \{1, 6, 11\},
 \{1, 5, 6\},
 \{1, 8, 11\}$\end{tabular} \\
\bottomrule
\end{tabular} 
\end{table}

\end{document}